\documentclass[12pt,psfig]{article}
\usepackage{amsthm,amssymb,amsmath,amsfonts,amscd}
\usepackage{graphicx,epsfig,latexsym}
\usepackage{cite,calc,xcolor,subfigmat,mathrsfs,dsfont}
\usepackage{hyperref}

\newcommand{\F}{\mathscr{F}}
\newcommand{\LST}{\mathscr{L}}
\newcommand{\E}{\mathscr{E}}
\newcommand{\N}{\mathbb{N}}
\newcommand{\ta}{\mathscr{T}}
\newcommand{\Le}{\left}
\newcommand{\Ri}{\right}
\newcommand{\NZ}{\mathbb{N}_0}
\newcommand{\Span}{{\rm span}}
\newcommand{\ad}{{\rm ad}}
\newcommand{\rank}{{\rm rank}}
\newcommand{\0}{\mathbf{0}}
\newcommand{\x}{\mathbf{x}}
\newcommand{\y}{\mathbf{y}}
\newcommand{\m}{\mathbf{n}}
\newcommand{\n}{\mathbf{n}}
\newcommand{\ie}{{\em i.e.,} }
\newcommand{\eg}{{\em e.g.,} }
\newcommand{\hot}{{\rm h. o. t. }}
\newcommand{\ddx}{\frac{\partial}{\partial x}}
\newcommand{\ddy}{\frac{\partial}{\partial y}}
\newcommand{\ddz}{\frac{\partial}{\partial z}}
\newtheorem{thm}{Theorem}[section]
\newtheorem{cor}[thm]{Corollary}
\newtheorem{lem}[thm]{Lemma}
\newtheorem{prop}[thm]{Proposition}
\theoremstyle{definition}
\newtheorem{defn}[thm]{Definition}

\newtheorem{rem}[thm]{Remark}

\numberwithin{equation}{section}
\def\Blem {\begin{lem}}
\def\Elem {\end{lem}}
\def\be {\begin{equation}}
\def\ee {\end{equation}}
\def\ba {\begin{eqnarray}}
\def\ea {\end{eqnarray}}
\def\bes {\begin{equation*}}
\def\ees {\end{equation*}}
\def\bas {\begin{eqnarray*}}
\def\eas {\end{eqnarray*}}
\def\bpr {\begin{proof}}
\def\epr {\end{proof}}

\begin{document}
\baselineskip=18pt
\renewcommand {\thefootnote}{ }

\pagestyle{empty}

\begin{center}
\leftline{}
\vspace{-0.500 in}
{\Large \bf Bifurcation control and universal unfolding for Hopf-zero singularities with leading solenoidal terms } \\ [0.3in]

{\large Majid Gazor$^{a,b}$\footnote{$^a\,$Corresponding author. Phone: +98(31) 33913634; Fax: +98(31) 33912602; Email: mgazor@cc.iut.ac.ir; n.sadri@math.iut.ac.ir.} and Nasrin Sadri\(^{a}\)}

\vspace{0.105in} {\small {\em $^{a}$Department of Mathematical Sciences,
Isfahan University of Technology
\\[-0.5ex]
Isfahan 84156-83111, Iran\\
$^{b}$School of Mathematics,
Institute for Research in Fundamental Sciences (IPM),\\
P.O. Box: 19395-5746, Tehran, Iran }}

\vspace{0.15in}

{ \bf { Dedicated to Professor James Murdock on the occasion of his 70th birthday }}

\vspace{0.05in}

\vspace{0.15in}



\vspace{0.05in}

\noindent
\end{center}

\vspace{-0.10in}

\baselineskip=16pt

\:\:\:\:\ \ \rule{5.88in}{0.012in}

\begin{abstract}

In this paper we introduce {\it universal asymptotic unfolding normal forms} for nonlinear singular systems. Next, we propose an approach to find the parameters of a parametric singular system that they play the role of the universal unfolding parameters. These parameters {\it effectively} influence the local dynamics of the system. We propose a systematic approach to locate local bifurcations in terms of these parameters. Here, we apply the proposed approach on Hopf-zero singularities whose the first few low degree terms are incompressible. In this direction, we obtain novel orbital and parametric normal form results for such families by assuming a nonzero quadratic condition. Moreover, we give a truncated universal asymptotic unfolding normal form and prove the finite determinacy of the steady-state bifurcations for two most generic subfamilies of the associated {\it amplitude} systems. We analyze the local primary bifurcations of equilibria and limit cycles, and the secondary Hopf bifurcation of invariant tori. The results are successfully implemented and verified using {\sc Maple}. By employing the proposed approach, we design an effective multiple-parametric quadratic state feedback controller for a singular system on a three dimensional central manifold with two imaginary uncontrollable modes. We illustrate how our program systematically identifies the distinguished (universal unfolding) parameters, derives the estimated transition varieties in terms of these parameters, and locates the local primary and secondary bifurcations of equilibria, limit cycles and invariant tori. This approach is useful in designing efficient nonlinear feedback controllers (single or multiple inputs) for local bifurcation control in engineering problems.

\vspace{0.10in} \noindent {\it Keywords:} \ Bifurcation control; universal asymptotic unfolding; primary and secondary bifurcations; Hopf-zero singularity; solenoidal vector fields.

\vspace{0.10in} \noindent {\it 2010 Mathematics Subject Classification}:\, 34C20; 34A34.

\noindent \rule{5.88in}{0.012in}
\end{abstract}

\vspace{0.2in}

\section{Introduction }

Any small perturbation of a singular differential system may substantially change the qualitative dynamics of the system. Therefore for such engineering singular problems, a practical approach is to design a controller to inhibit its real world dynamics rather than being dominated by the mathematical modeling imperfections. This can be achieved by computing the {\it universal unfolding} of the differential system and be used for a {\it bifurcation controller design}. Bifurcation control has many applications in engineering problems such as power, electronics, and mechanical systems; predicting and preventing voltage collapse and oscillation in power networks, high-performance circuits and oscillator designs; \eg see \cite{ChenBifuControl,ChenBifControl2000}. The idea here is to design a controller for a nonlinear system so that the system follows a certain bifurcation branch and thus, it behaves as desired. Recently, normal form theory has been used for local bifurcation control; see \cite{ChenBifuControl,KangKrener,KangIEEE,Kang98,Kang04}. In this paper we describe how parametric normal form theory can propose effective feedback controller designs for an engineering problem. We apply it to Hopf-zero singularity, \ie
\be\label{Eq1}
\dot{x}:= f(x, y, z), \quad \dot{y}:=z+ g(x, y, z), \quad \dot{z}:=-y+ h(x, y, z),
\ee for \((x,y,z)\in \mathbb{R}^{3},\) where \(f, g, h\) do not have linear and constant terms. Additionally, we assume that certain first few low degree terms of equation \eqref{Eq1} constitute a solenoidal vector field (that is, the case \(I\) in \eqref{Cases}). Here, a solenoidal vector field refers to a volume-preserving vector field. Throughout this paper, we will interchangeably use the terms {\it vector field}, denoted by \(f\ddx+(z+g)\ddy+(-y+h)\ddz\) or \((f, z+g, -y+h),\) and the {\it differential system} \eqref{Eq1}.

\pagestyle{myheadings} \markright{{\footnotesize {\it M. Gazor and N. Sadri \hspace{2.2in} Bifurcation control and universal unfolding  }}}

Normal form computations of vector fields can not be performed up to an infinite degree Taylor expansion and thus, truncation at a finite degree is unavoidable. When the truncated normal form and the original system have the same {\it qualitative properties}, the system is called {\it finitely determined}; see
\cite[Page vi]{MurdBook}. We recall that {\it qualitative properties} are defined as invariances of a given equivalence relation. Thus, a system is finitely
determined when it is equivalent to one of its truncated normal forms; \eg see \cite[Fold bifurcation (3.1), Lemma 3.1]{Kuznetsov} for an easy example.
The dynamics of a Hopf-zero singular system given by equation \eqref{Eq1} may not be finitely determined with respect to topological equivalence (see \cite[Definition 2.15]{Kuznetsov} and \cite[Page 191]{ChowBook}) and many dynamical properties such as {\it heteroclinic orbit breakdowns} and {\it \v{S}il'\'{n}ikov bifurcations} can not be detected through truncated normal form computations; \eg see
\cite{DumortierHopfZ,SearaHopfZero,BroerHopfZero}. However, normal form computations are still useful for the analysis of
{\it finitely determined dynamical properties}. Finite determinacy naturally raises notions of {\it \(n\)-equivalence relation} and the {\it \(n\)-universal asymptotic unfolding} defined by Murdock \cite{MurdBook,Murd98}.

Murdock \cite{MurdBook,Murd98} defines two systems as {\it \(n\)-equivalent} when they share all their \(n\)-jet (\(n\)-degree Taylor expansion) determined properties. The \(n\)-asymptotic unfolding is defined based on the \(n\)-equivalence relation and it seems the most natural way of defining a versal unfolding amenable to computation and normal form analysis; \eg see \cite[Theorem 1]{Murd09}. We slightly modify versal asymptotic unfolding and call it {\it versal asymptotic unfolding normal form}, that is, a versal asymptotic unfolding for our simplest (orbital) normal form system. This also exhibits the \(n\)-jet determined properties of all small perturbations of the original system.

In order to illustrate the dynamics invariant (or those not invariant) under \(n\)-equivalence relation, in subsection \ref{SubsecFinite}, we prove that bifurcations and stabilities of equilibria and small limit cycles for a generic subfamily are invariant under \(2\)-equivalence relation while the secondary Hopf bifurcation of invariant torus is invariant under \(3\)-equivalence relation. Moreover, two more degenerate subfamilies are also discussed.
The proof uses contact-equivalence relation (see \cite[Page 166]{GolubitskyStewartBook} and \cite{GazorKazemi}) and is based on computations of certain modules, called {\it high order term} modules. High order term module of a map \(v\) is a module over the ring of all scalar maps (germs) and is generated by all monomial vectors like \(p\) so that \(v+p\) is contact-equivalent to \(v.\)
We further discuss bifurcation of a heteroclinic cycle surrounding a continuous family of invariant tori for a \(2\)-jet normal form system; see subsection \ref{BifAnala}. Detection of the saddle-saddle connection simply uses the first integral of the truncated normal form and demonstrates an important advantage of our normal form representation in conservative-nonconservative terms. This heteroclinic cycle is not invariant under \(2\)-equivalence relation and thus, it will not be pursued in our bifurcation control; also see \cite[Page 225]{LangfHopfSteady}.

A parametric vector field \(v(\x, \mu)\) (for \(\mu\in \mathbb{R}^p\) from a small neighborhood of the origin and \(\x\in \mathbb{R}^3\)) is called a {\it perturbation} or an {\it unfolding} for \(w(\x)\) when \(v(\x, \0)= w(\x)\). In other words, \(v(\x, \mu)= w(\x)+p(\x, \mu)\) where \(p(\x, \0):= \0.\)
The parameterized family of perturbations
\bes v(\x, \mu):= w(\x)+ \sum_{\{i\,|\, |\m_i|\leq n\}} \mu_{i}\x^{\m_i}, \hbox{ for } \m_i=(m_1, m_2, m_3)\in (\mathbb{N}\cup\{\0\})^3, |\m_i|=m_1+m_2+m_3, \ees is called the \(n\)-{\it jet general perturbation} for \(w(\x)\), that is \(w(\x)\) plus all monomial perturbations of degree less than or equal to \(n.\) In order to simplify such systems, we consider a smooth locally invertible change of state variables
\be\label{Changx}\x=\phi(\y, \mu), \hbox{ where the Jacobian } D_\y\phi \hbox{ at the origin is invertible and } \phi(\0, \mu)=\0, \ee
and time rescaling
\be \label{Tchang}\tau= T(\y, \mu) t, \hbox{ with } T(\0, \0)\neq0, \ee
where \(t\) and \(\tau\) denote old and new time indices. Using \eqref{Changx} and \eqref{Tchang}, we may transform \(v(\x, \mu)\) into an {\it orbitally equivalent vector field} as
\be\label{Orb}\tilde{v}(\y, \mu):=T(\y, \mu)(D_\y\phi)^{-1} v(\phi(\y, \mu), \mu).\ee
Parametric vector fields will be considered modulo an equivalence, a transformation described above in \eqref{Orb}.
The idea is to express such family into its simplest normal form; that is, an equivalent vector field with the least possible number of monomial {\it terms} in its \(n\)-degree truncated (\(n\)-jet) Taylor-expansion for any \(n\). It is known that the {\it simplest} orbital normal form is {\it unique} (\ie the normalized coefficients are uniquely determined in terms of the original system) when a {\it formal basis style} is chosen.

We recall that a {\it style} is a rule on how to choose a complement to the {\it removable space}, and the {\it removable space} is defined as the space spanned by all terms that can be simplified from the system using the equivalence \eqref{Orb}. Thus, normal form style determines what terms are simplified and what terms shall remain in the normal form system. The inner product, semisimple, \(\mathfrak{sl}(2)\), simplified and formal basis styles are among the main examples; see \cite[Page ix]{MurdBook} and \cite{Murd16,GazorYuSpec}. A {\it formal basis style} basically determines the priority of elimination between different alternative omittable terms. More precisely, in normal form computations there exist alternative monomial terms for elimination from a given vector field and the rule on how to choose these alternative terms is called {\it formal basis style}. Formal basis style is determined by an ordering on a formal basis for all vector fields (\eg all monomial vector fields) so that those succeeding others are in priority of elimination; see \cite{GazorYuFormal,GazorYuSpec}. We assume that a formal basis style has been fixed.

From now on, we use the word {\it ``equivalence''} instead of the word {\it ``orbital equivalence''} given by equation \eqref{Orb}, unless it is stated otherwise. Recall that an \(n\)-jet of a vector field refers to its \(n\)-degree Taylor expansion.

\begin{defn}
A parametric vector field \(v(\x, \nu)\) is called an \(n\)-{\it versal asymptotic unfolding normal form} for \(w(\x)\) if
\begin{enumerate}
  \item[a.]\label{itm1} for each small perturbation \(w(\x)+p(\x, \epsilon)\) of \(w(\x),\) there exists
  a polynomial map \(\nu(\epsilon)\) such that \(\nu(\0)=\0\) and \(w(\x)+p(\x, \epsilon)\) is equivalent to \(v(\x, \nu(\epsilon))\) modulo degrees of higher than or equal to \(n+1\).
  \item[b.] the \(n\)-jet of \(w(\x, \0)\) is the \(n\)-jet of the simplest orbital normal form for \(v(\x)\).
\end{enumerate}
We call a parametric system \(v(\x, \nu),\) an {\it \(n\)-universal asymptotic unfolding normal form} for \(w(\x)\) when
\begin{itemize}
  \item \(v(\x, \nu)\) is an \(n\)-versal asymptotic unfolding normal form for \(w(\x)\).
  \item The \(n\)-jet of \(v(\x, \nu)\) is the \(n\)-jet of the simplest orbital normal form for the \(n\)-jet general perturbation of \(w(\x)\).
\end{itemize} Then, we refer to the parameter \(\nu\) by {\it universal unfolding parameter}; also see \cite{DumortierSingularities,AnosovArnold} and \cite[Definition 2.18]{Kuznetsov}.
\end{defn}

\begin{thm}
For any vector field \(w(\mathbf{x})\) and any given natural number \(n\), there always exists an \(n\)-universal asymptotic unfolding normal form \(v(\x, \nu).\) Besides, for any perturbation \(w(\x)+p(\x, \epsilon)\), the map \(\nu(\epsilon)\) in item (a) is unique modulo degrees that do not affect \(n\)-jet of  \(v(\x, \nu)\).
\end{thm}
\bpr As we have claimed in \cite{GazorYuSpec}, the orbital normal form version of \cite[Lemma 4.3]{GazorYuSpec} also holds. Hence once a formal basis style is chosen, the simplest orbital normal form of the \(n\)-jet general perturbation of \(w(\x)\) exists and thus, its \(n\)-jet gives rise to \(n\)-universal asymptotic unfolding normal form \(v(\x, \nu).\) The orbital normal form of a given perturbation \(w(\x)+p(\x, \epsilon)\) readily gives the polynomial map \(\nu(\epsilon)\) in item (a). In order to achieve item (b), one needs to choose formal basis style appropriately. In fact the priority of elimination should be effectively given to parameter-independent terms rather than those terms depending to \(\mu_i\). The claim about uniqueness of \(\nu(\epsilon)\) directly follows from the uniqueness of simplest orbital normal form coefficients.
\epr

The universal asymptotic unfolding normal form facilitates the (finitely determined) local bifurcation analysis in terms of the unfolding parameters. However in practical engineering problems, the mathematical models are mostly involved with parameters such as control parameters; see \cite{ChenBifuControl}. Therefore, a practically useful approach needs to locate the bifurcations in terms of the original (control) parameters of a parametric (control) system. This has been rarely performed in the existing normal form literature; see \cite{YuNonlinearity} and
\cite[Pages 99-126]{ChenBifuControl} for Hopf bifurcation control. Hence, a useful normal form analysis needs to compute the relations between the unfolding parameters and the parameters of the original system. This reveals the impact of control parameters on our nonlinear control system. More details have been given in section \ref{secPNF} and the results have successfully been implemented for Hopf-zero singularities with dominant solenoidal terms in our {\sc Maple} program. In order to achieve this goal, we first need to derive orbital normal form of equation \eqref{Eq1} and then, parametric normal forms of its (small) multiple-parametric perturbations (perturbations depending on several variables).

The only existing normal form literature on hypernormalization (simplification beyond classical normal forms) of Hopf-zero singularity is due to the authors of \cite{ChenHopfZ,YuHopfZero,AlgabaHopfZ}. All of these papers have assumed two non-zero quadratic conditions (along with other non-resonance conditions for degrees higher than two) given by
\be\label{Alg} f_{xx}(\0)\left(g_{yx}(\0)+h_{zx}(\0)\right) \neq0. \ee This paper aims to complete the results on hypernormalization of Hopf-zero singularity with non-zero quadratic part. Therefore, throughout this paper we assume a non-zero condition as
\be\label{a01}f_{yy}(\0)+f_{zz}(\0)\neq0,\ee also see \cite{GazorMokhtari}. Note that right hand side of inequalities \eqref{Alg} and \eqref{a01} represent certain coefficients from the classical normal forms. Equation \eqref{Eq1} is equivalent to the following normal form (in cylindrical coordinates) (see \cite[Lemma 3.1]{GazorMokhtari})
\be\label{First2ndlevel}
\dot{x}= \rho^{2}+\sum^{\infty}_{i=2} a_{i} x^i, \quad \dot{\rho}=\sum^{\infty}_{i=0} b_{i} x^{2i+1}\rho, \quad \dot{\theta}= 1+\sum^{\infty}_{i=0} c_{i} x^{2i+1}.
\ee
We further assume that there exists \(a_{k}\neq0\) for some \(k.\) Now define
\be\label{sr}
r:=\min\{i\mid a_{i}\neq 0,i\geq 1\} , \quad s:=\min\{j\mid b_{j}\neq 0,j\geq 1\}.
\ee
In this paper we compute the orbital normal forms of the system \eqref{Eq1}, by assuming \eqref{a01} and
\be\label{rs} r<s,\ee
and also parametric normal forms for any of its multiple-parametric perturbations. Here, we assume that \(s<\infty.\) The case \(s=\infty\)
consists of all solenoidal Hopf-zero vector fields and is discussed in \cite{GazorMokhtariInt}. The results on the other two cases (\(r>s\) and \(r=s\)) are in progress and appear elsewhere.

Assuming that \eqref{a01} and \eqref{rs} hold, we prove in Theorem \ref{ONFs+1level} that the system \eqref{First2ndlevel} is equivalent to
\bes
\dot{x}=2\rho^2+ a_r x^{r+1}+\sum^{\infty}_{k=s}\beta_kx^{k+1}, \quad
\dot{\rho}=-\dfrac{a_r (r+1)}{2}x^r\rho+\frac{1}{2}\sum^{\infty}_{k=s}\beta_{k}x^k\rho, \quad
\dot{\theta}=1+ \sum^{r}_{k=1}\gamma_{k}x^k,
\ees where \(\beta_{k}=0\) for \(k\equiv_{2(r+1)}-1\)  and \(k\equiv_{2(r+1)}s.\)
Finally, we prove that any multiple-parametric perturbation of the system \eqref{First2ndlevel} (for \(r<s\)) can be transformed into the \((s+1)\)-th level parametric normal form (Theorem \ref{MainThm} part I)
\ba
\dot{x}&=&2\rho^2+a_{r\0}x^{r+1}+\sum_{-1\leq i< r+1}{a_{i\n}}x^{i+1}\mu^{\n}+\sum_{0\leq i < s,i\neq r}{b_{i\n}} x^{i+1}\mu^{\n}+\sum_{k=s}\beta_{k\n}x^{k+1}\mu^{\n},\\\nonumber
\dot{\rho}&=&-\dfrac{a_{r\0}(r+1)}{2}x^r\rho+\sum_{-1\leq i< r+1}{a_{i\n}}(\dfrac{i+1}{2})\rho\mu^{\n}+\sum_{0\leq i < s,i\neq r}\frac{1}{2}{b_{i\n}} x^{i}\rho\mu^{\n}+\sum_{k=s}\frac{1}{2}\beta_{{k\n}}x^{k}\rho\mu^{\n},\\\nonumber
\dot{\theta}&=&1+ \sum^{r}_{k=1}\gamma_{k\n}x^k\mu^{\n},
\ea
and its \(s+1\)-universal asymptotic unfolding is given by (Theorem \ref{MainThm} part II)
\ba
\dot{x}&=& 2\rho^2+a_rx^{r+1}+\beta_sx^{s+1}+\sum_{1\leq i\leq r}\nu_{i}x^{i-1}+\sum_{i=r+1}^{N}\nu_{i}x^{k_{i}+1}, \\\nonumber
\dot{\rho}&=& -\dfrac{a_{r}(r+1)}{2}x^r\rho+\frac{1}{2}\beta_sx^{s}\rho-\sum_{1\leq i\leq r}\frac{(i-1)}{2}\nu_{i}x^{i-2}\rho+\frac{1}{2}
\sum_{i=r+1}^{N}\nu_ix^{k_i}\rho,
\\\nonumber
\dot{\theta}&=&1+\sum^{r}_{i=1}(\gamma_i+\omega_{i})x^{i}.
\ea Here, \(N=r+s-\lfloor\frac{s}{2(r+1)}\rfloor,\) \(\n=(m_1,\ldots, m_p)\in(\N\cup\{0\})^p,\) the parameters \(\mu:=(\mu_1, \ldots, \mu_p)\in \mathbb{R}^p,\) \(\nu_i, \omega_{i}\in \mathbb{R},\) \(\mu^\n:= \mu_1^{m_1}\ldots\mu_p^{m_p},\) and the coefficients \(a_{i\n}, b_{i\n}, \beta_{i\n}, \gamma_{i\n}, \beta_s, \gamma_i\in \mathbb{R}.\)
The family associated with assumptions \eqref{a01} and \eqref{rs} are large enough so that they may appear near stagnation points associated with perturbations of incompressible fluid flows, three dimensional magnetic field lines and well-known systems such as Michelson system.

The rest of this paper is organized as follows. In section \ref{sec2} we provide the time rescaling structure and introduce our normal form style.
Section \ref{sec3} presents our orbital normal form results. Parametric and universal asymptotic unfolding normal forms for \(s<\infty\) are
given in section \ref{secPNF}. Normal forms in cylindrical coordinates have a phase component (\ie \(\dot{\theta}\)-component in equation \eqref{First2ndlevel}) associated with angle coordinate and it is a common approach to ignore the phase component for bifurcation analysis. The obtained planar system is called {\it amplitude system}. In section
\ref{SecBF} using contact-equivalence relation, we prove that steady-state bifurcations associated with the amplitude systems for \(r:=1\) and \(2\) are
finitely determined.
Then by assuming that \(r=1\) and \(2\) we discuss a limited (not a complete) bifurcation analysis of the universal asymptotic unfolding normal forms. In section \ref{SecBC} we explain how to identify the {\it parameters} of a parametric system, \ie those playing the role of asymptotic universal unfolding parameters. Roughly speaking, cognitive choices of these parameters effectively control finitely determined local dynamics of the system such as primary and secondary bifurcations of equilibria, limit cycles and invariant tori. We have implemented our approach in Maple so that it derives the transitions sets in terms of original parameters and symbolic coefficients. As far as our information is concerned this is new in the literature of both normal form and bifurcation theory.
Finally, section \ref{secExm} applies our approach to an illustrating example with two imaginary uncontrollable modes. Here, estimated transition sets are drawn in terms of the distinguished parameters and are supported with some numerical simulations. This demonstrates that our distinguished parameters can suitably control the local dynamics of a nonlinear Hopf-zero singular system and can be used for a possible engineering controller design.

\section{ Time rescaling structure and normal form style }\label{sec2}

We follow \cite{GazorMokhtari} and define the vector fields
\bas
F^{l}_{k}&:=&(k-l+1)x^{l+1}\rho^{2(k-l)}\frac{\partial}{\partial x}-\left(\frac{l+1}{2}\right)x^{l}\rho^{2k-2l+1}\frac{\partial}{\partial \rho},\\
E^{l}_{k}&:=&x^{l+1}\rho^{2(k-l)}\frac{\partial}{\partial x}+\frac{1}{2}x^{l}\rho^{2k-2l+1}\frac{\partial}{\partial \rho},\\
\Theta^{l}_{k}&:=&-x^{l}\rho^{2(k-l)}\frac{\partial}{\partial \theta},
\eas here the variables \(x, \rho\) and \(\theta\) represent cylindrical coordinates. Note that \(\rho\) does not take negative values and otherwise, it is singular at \(\rho=0\). Then, by \cite[Theorem 2.4]{GazorMokhtari} any Hopf-zero normal form system \eqref{First2ndlevel} in cylindrical coordinates can be expanded in terms of \(F, E, \) and \(\Theta\)-terms like
\be\label{Theta}
v:= \Theta^0_0+ \sum a_{i,j}F^i_j+ \sum b_{i,j}E^i_j+ \sum c_{i,j}\Theta^i_j.
\ee We denote \(\LST\) for the vector space generated by all Hopf-zero normal forms expanded with respect to \(F, E, \) and \(\Theta\)-terms.
The Lie algebra structure constants follow \cite[Lemma 2.3]{GazorMokhtari}.

We define a module structure for \(\LST\) that is instrumental for computing the effect of the near-identity time rescaling.
The integral domain of formal power series (denoted by \(\mathcal{R}\)) generated by monomials
\be\label{tmn}
Z^{m}_{n}:=x^{m} \rho^{2(n-m)}
\ee represents the space of time-rescaling generators; \(Z^m_n\) generates \(t:=(1+x^{m} \rho^{2(n-m)})\tau\), where \(t\) and \(\tau\) denote the
old and new time variables. Hence, \(\mathcal{R}:= \Span \{Z^m_n\,|\, m\leq n, m, n \in \NZ=\N\cup\{0\}\}\) acts on \(\LST\) and \(\LST\) is an \(\mathcal{R}\)-module.

\begin{lem}\label{structure} The \(\mathcal{R}\)-module structure constants associated with time rescaling are given by
\bas
Z^{m}_{n}F^{l}_{k}&=&\dfrac{k+2}{k+n+2}F^{l+m}_{k+n}+\dfrac{m(k+2)-n(l+1)}{k+n+2}E^{l+m}_{k+n},\\
Z^{m}_{n}E^{l}_{k}&=&E^{l+m}_{k+n},\\
Z^{m}_{n}\Theta^{l}_{k}&=&\Theta^{l+m}_{k+n}.
\eas
\end{lem}
In this paper we merely apply the time rescaling space
\be\label{ta0} \ta:= \Span \{Z^m_n\in \ta\,|\, \hbox{ for } m= n, \hbox{ or } m+1=n\},\ee and use the formulas and identical notations from \cite{GazorMokhtari,GazorMoazeni}. Our computations and {\sc Maple} program suggest that other time rescaling generators (from \(\mathcal{R}\setminus\ta\)) do not simplify the system beyond what we present in this paper.

Any normal form computation requires a normal form style. Recall that for the cases with alternative terms for elimination, a formal basis style determines the priority of elimination. This is determined by an ordering on basis terms (\(F, E, \Theta\)-terms) of \(\LST\) in formal basis style; see
\cite[Page 1006]{GazorYuSpec}. To determine our ordering, we need a grading function \(\delta.\) The grading function decomposes the space \(\LST\) into \(\delta\)-homogeneous vector subspaces \(\LST_i\) spanned by terms of the same grade, say grade \(i\). The decomposition \(\LST=\sum \LST_i\) builds a grading structure for \(\LST\) and this must make \(\LST\) a graded Lie algebra, \ie \([\LST_i, \LST_j]\subseteq \LST_{i+j}\); see
\cite[Page 1006]{GazorYuSpec}. Assuming that \(\delta\) is given (\eg see equations \eqref{delta1} and \eqref{Pdelta}), for any \(v, w\in \{F^m_n, E_k^l, \Theta^p_q\}\) we define \(v\prec w,\) when
\begin{itemize}
  \item Style:
\(\left\{
    \begin{array}{ll}
    \delta(v)<\delta(w),&\\
      \delta(v)=\delta(w), & v=F^m_n \hbox{ and } w=E_k^l \hbox{ or } w=\Theta^p_q,\\
      \delta(v)=\delta(w), & v=E_k^l \hbox{ and } w=\Theta^p_q.
    \end{array}
  \right.\)
\end{itemize} This indicates that our priority of elimination is with low grade terms over higher grades and then, \(F\)-terms over \(E\)-terms.
We denote
$(a)^{k}_{b}:=a(a+b)(a+2b)\cdots(a+(k-1)b)$ for any natural number $k$ and real number $b,$ and for
any integer numbers \(m, n, p\) the notation \bes m\equiv_{p}n\ees is used when there exists an integer \(k\) such that \(m-n=kp.\)

\section{ The orbital normal forms }\label{sec3}

Equation \eqref{Eq1} can be transformed into the normal form equation \eqref{First2ndlevel} or equivalently,
\be\label{2ndlevel}
v^{(2)}:= \Theta^0_0+ a_0F^{-1}_0+ \sum a_iF^i_i+ \sum b_iE^i_i+ \sum c_i\Theta^i_i,
\ee where \(a_i, b_i, c_i \in \mathbb{R}\); see \cite[Lemma 3.1]{GazorMokhtari}. Without loss of generality we may assume that \(a_0=1\); see the comments above \cite[Remark 3.2]{GazorMokhtari}. (Note that the orbital equivalence may here include reversal of time.) Recall \(r\) and \(s\) from equation \eqref{rs}.
Orbital normal form reduction of equation \eqref{2ndlevel} is split into three cases (this is similar to the three cases of Bogdanov-Takens singularity \cite{KokubuWang,GazorMoazeni,BaidSand91})
\be\label{Cases}
\hbox{ Case I: }\; r<s, \quad \hbox{Case II: }\; r> s, \quad \hbox{ Case III: }\; r=s.
\ee In this paper we only deal with the case I. Recall the grading function (see \cite[Equation 4.3]{GazorMokhtari})
\ba\label{delta1}
\delta(F^{l}_{k})=\delta(E^{l}_{k})=r(k-l)+k ,\quad  \delta(\Theta^{l}_{k})=r(k-l)+k+r
\ea and the linear map
\ba\label{differential}
d^{n, N}\big(S_{n-N+1}, \ldots, S_{n-r}; T_{n-N+1}, \ldots, T_{n-r}\big)&:=&\sum^{N-1}_{k=r} \big([S_{n-k}, v_{k}]+ T_kv_{n-k}\big),
\ea for any \((S_{n-N+1}, \ldots, S_{n-1-r}; T_{n-N+1}, \ldots, T_{n-1-r})\in \ker d^{n-1, N-1},\) the updating vector field \(v=\sum^\infty_{n=r} v_n,\) \(S_i\in \LST_i,\) and \(T_i\in \ta_i\), where \(\LST_i, \ta_i\) denote the \(\delta\)-homogenous subspaces of \(\LST\) and \(\ta.\) Note that the vector field \(v\) is sequentially being updated in the process of normal form computation. We skip many subtleties of the subject; see \cite{benderchur,Sanders03,GazorYuFormal,GazorYuSpec,GazorMoazeni} for more information.

\begin{prop} Consider our formal basis style, the grading function \(\delta\) in \eqref{delta1}, and the linear map \(d^{n, N}\) in equation \eqref{differential}. Then, any Hopf-zero vector field \(v\) in equation \eqref{Eq1} can be transformed into an equivalent \((N+1)\)-th level normal form \(v^{(N+1)}=\sum w_n\) so that \(w_n\) belongs to the complement space of \({\rm im}\, d^{n,N},\) the complement space is uniquely obtained according to the normal form style.
\end{prop}
\bpr The proof follows from the fact that the hypotheses of \cite[Theorem 6.11]{benderchur} hold; also see \cite[Theorem A.1]{Sanders03} and
\cite[Lemma 4.3]{GazorYuSpec}.
\epr
Denote
\bes\Bbb{F}_{r}:=F^{-1}_{0}+a_{r}F^{r}_{r}.\ees
We now provide some technical formulas for obtaining orbital normal forms.
\begin{lem}\cite[Lemma 3.3]{GazorMokhtariInt} For nonnegative integers \(m, n,\) \(n\geq m,\) the vector field
\bes \mathcal{T}^{m}_{n}:=\sum^{n-m-1}_{l=0}\dfrac{{a_{r}}^{l}\big((n-m-1)(r+1)-m-1\big)^{l}_{-2(r+1)}}{2^{l+1}\big(m+1\big)^{l+1}_{r+1}}\Theta^{m+lr+l+1}_{n+lr}
\ees satisfies
\bes
[\mathcal{T}^{m}_{n},\Bbb{F}_{r}]+\Theta^{m}_{n}=\dfrac{{a_{r}}^{n-m}\big((n-m-1)(r+1)-m-1\big)^{n-m}_{-2(r+1)}}
{2^{n-m}\big(m+1\big)^{n-m}_{r+1}}\Theta^{nr-mr+n}_{nr-mr+n}.
\ees
\end{lem}

\begin{cor}\label{lemm23}
For any nonnegative integer \(m,\) we have
\be
Z^{m}_{m+1}\Theta^{0}_{0}+\left[\frac{1}{2(m+1)}\Theta^{m+1}_{m+1},\Bbb{F}_{r}\right] =\dfrac{-a_{r}}{2}\Theta^{m+r+1}_{m+r+1}.
\ee Besides, for each \(m\neq 0\) the following vector field
\be\label{Zcal}
\mathcal{Z}^m_m:=\frac{1}{(m+1)^{2}_{1}}F^{m}_{m}+\frac{1}{(m+2)}E^{m}_{m}
\ee satisfies
\be\label{Zmm}
Z^{m}_{m}\Bbb{F}_{r}+ \left[\mathcal{Z}^m_m, \Bbb{F}_{r}\right] = \dfrac{a_{r}(r+1)}{m+1} F^{m+r}_{m+r}.
\ee
\end{cor}

\begin{proof}
By Lemma \ref{structure} we have
\bas
Z^{m}_{m}F^{-1}_{0}&=&\dfrac{2}{m+2}F^{m-1}_{m}+\dfrac{2m}{m+2}E^{m-1}_{m},\\
Z^{m}_{m}F^{r}_{r}&=&\dfrac{r+2}{m+r+2}F^{m+r}_{m+r}+\dfrac{m}{m+r+2}E^{m+r}_{m+r}.
\eas The proof is complete by a straightforward computation.
\end{proof}

We have already used the following lemma in \cite{GazorMokhtariEul,GazorMokhtari,GazorMokhtariInt}. Since it plays a central role in efficient use of time rescaling for simplifying \(\Theta\)-terms, we here state it as a Lemma.

\Blem\label{theta00} The transformation
\be
[x(t), \rho(t), \Theta(t)] = \varphi(x(t), \rho(t), \theta(t)):= [x, \rho, \theta-t],
\ee where \(\Theta\) and \(\theta\) are the new and old phase variables, is an invertible linear change of state variables so that \(\varphi\) transforms away the linear part \(\Theta^0_0\) from the normal form system \eqref{Theta}. In addition, \(\varphi^{-1}\) adds \(\Theta^0_0\) back into the normal form system.
\Elem
\bpr
The claim is true since the \(\Theta\)-component is decoupled from \(x\) and \(\rho\)-components; also see the comments on
\cite[Pages 317-318]{GazorMokhtari}.
\epr

The first part of Corollary \ref{lemm23} implies that we can use the time rescaling associated with \(Z^{m}_{m+1}\) for eliminating \(\Theta^{m+r+1}_{m+r+1}\)-terms. This is so when \(\Theta^0_0\) appears in the vector field. On the other hand, the second part of
Corollary \ref{lemm23} implies that time rescaling terms associated with \(Z^m_m\) can be used to simplify \(F^{m+r}_{m+r}\)-terms.
However, applying time rescaling \(Z^m_m\) on a vector field, whose expansion includes \(\Theta^0_0\), creates terms of the form \(\Theta^m_m\) which had been simplified from the system in earlier steps using \(Z^{m-r-1}_{m-r}.\) Therefore, it is beneficial to transform \(\Theta^0_0\) away from the system when we apply time rescaling \(Z^m_m\). Further, we transform it back to the system once we intend to apply \(Z^{m}_{m+1}\)-type of time rescaling.

\begin{cor} Let
\bes
\mathcal{Z}^m_{m+1,r}:=\dfrac{1}{(m+1)(m+3)}F^m_{m+1}+\dfrac{1}{(m+3)}E^m_{m+1}+\dfrac{a_r (r+1)}{2 (m+1)(m+r+2)}F^{m+r+1}_{m+r+1}.
\ees Then, for each \(m\neq 0\) we have
\bas Z^m_{m+1}\mathbb{F}_r+[\mathcal{Z}^m_{m+1,r}, \mathbb{F}_r]&=&\dfrac{{a_r}^2}{2}\Big(\dfrac{ (m-2r-1)}{(m+3)(m+r+2)}+\dfrac{ (m-r-1)}{(m+2r+3)}\\
&&-\dfrac{ r(r+2)(m+1)}{ (m+3)(m+r+2)(m+r+3)}\Big)F^{m+2r+1}_{m+2r+1}.\eas
For the case \(m=0,\) the following vector field
\bas
\mathcal{Z}^{0}_{1, r}&:=&\frac{1}{3} F^0_1+\dfrac{a_{r}(2r^2+10r+9)}{6(r+2)(r+3)} F_{r+1
}^{r+1}-\frac{a_{{r}}}{2(r+3)} E_{r+1}^{r+1}.
\eas satisfies
\ba\label{Z01}
Z^0_1\Bbb{F}_{r}+[{\mathcal{Z}^{0}_{1, r}},\Bbb{F}_{r}]=-\dfrac{{a_{r}}^{2}(7r^2+17r+9)}{6(r+2)(2r+3)}
F^{2r+1}_{2r+1}+\dfrac{{a_{r}}^2(r+1)}{2(2r+3)}E^{2r+1}_{2r+1}.
\ea
\end{cor}
\bpr This is a straightforward computation following Lemma \ref{structure} and \cite[Lemma 2.3]{GazorMokhtari}.
\epr

In the following we use the notations \({\mathcal{X}_{r}^{k}},\) \(\mathcal{F}^{-1}_{k,r},\) \(\mathcal{T}^0_{k,r},\) and \(\mathcal{E}^{0}_{k,r}\) along with a sequence \(e_{k,m}\) which are defined by \cite[Equations 3.4]{GazorMokhtariInt} and \cite[Equation 4.8]{GazorMokhtari}.

\begin{thm}\label{r+1level}
The \((r+1)\)-th level orbital normal form of \eqref{First2ndlevel} is
$$v^{(r+1)}:=\Bbb{F}_{r}+\sum^{\infty}_{k=s}\beta_{k}E^{k}_{k}+\sum^{r}_{k=0}\gamma_{k}\Theta^{k}_{k}$$
and $\beta_{k}=0$ for $k\equiv_{2(r+1)}-1$ when \(k\geq s.\)
\end{thm}
\bpr
Corollary \ref{lemm23} implies that \(\Theta^{m}_{m}\in {\rm im}\, d^{m+s, r+1}\) and \(F^{m}_m\in {\rm im}\, d^{m, r+1}\) for any \(m>r.\)
By \cite[Lemma 4.1, Equation (4.10)]{GazorMokhtari} we have
\ba\label{E01}
[\mathcal{E}^{0}_{2k+1,r},\Bbb{F}_{r}]&=&\dfrac{-{a_{r}}^{2k+2}(r+1)(2k+3)(2k+1)^{2k+1}_{-2}}{(2k)! 2^{2k+1}\big((2k+2)(r+1)+1\big)}E^{(2k+1)(r+1)+r}_{(2k+1)(r+1)+r}
\\\nonumber&&-2{a_{r}}^{2k+2}e_{2k+1,2k+2}\big(2(k+1)(r+1)+1\big)F^{(2k+1)(r+1)+r}_{(2k+1)(r+1)+r},
\ea where \(e_{2k+1,2k+2}\) is nonzero. Therefore, \(E^{(2k+1)(r+1)+r}_{(2k+1)(r+1)+r}\) also belongs to \({\rm im}\, d^{(2k+1)(r+1)+r, r+1}\) and the proof is complete.
\epr
The reader should note that the first non-zero coefficient \(b_j\) (for \(j\geq 1\)) from the system \eqref{First2ndlevel} may change in the \(r+1\)-th level normal form computation. Consequently, the number \(s\) defined in equation \eqref{rs} is changed and it must be updated once the \((r+1)\)-th level normal form coefficients are computed; also see \cite[Remark 2]{GazorMoazeni}. This is important for possible implementation of the results in a computer algebra system.
We recall that
\be\label{Eqn3.9}
\ker \ad_{F^{1}_{0}}\circ\ad_{\Bbb{F}_{r}} = \Span \left\{{\mathcal{X}_{r}^{k}}, \mathcal{F}^{-1}_{k,r}, \mathcal{E}^{0}_{k,r}, \mathcal{T}^0_{k,r}\,|\, k\in \N\right\},
\ee and assume that \(s<\infty.\)

Define \(\mathds{X}^k_r\) by
\bas
&& \sum^{k}_{m=0}\sum^{2k-m+2}_{l=0}\binom{k}{m}\dfrac {{a_{r}}^{m}b_{s}(s)_2^2(2k-m)g_{l}}{2k+mr+s}F^{(r+1)(l+m)+s}_{2(k+1)+r(m+l)+s}\\
&&+\sum^{k}_{m=0}\sum^{2k-m+2}_{l=0}\dfrac{{{a_{r}}^{m+l}b_{s}\binom{k}{m}(s)_2^2(2k-m)}\big(2(k-m-1)(r+1)-s\big)^{l}_{-2(r+1)}}
{2^{l+1}\big(2k+s+(l+m)r\big)\big(m(r+1)+s\big)^{l+1}_{r+1}}
E^{(r+1)(l+m)+s}_{2(k+1)+r(m+l)+s}.
\eas
\Blem\label{Lemm3.6} For any natural number \(k,\) we have
\((\mathcal{X}^{k-2}_r, \0, \mathds{X}^k_r; \0)\in \ker d^{2(k-1)(r+1)+s+r, s+1},\) where the zeros belong to \(\mathbb{R}^{s-r-1}\) and \(\mathbb{R}^{s+1-r},\) respectively.
\Elem
\bpr
Since
\be\label{Unpvd}
\sum^{k}_{m=0}\binom{k}{m}\dfrac{(2k-m)\big(2(k-m-1)(r+1)-s\big)^{2k-m-1}_{-2(r+1)}}{2^{2k-m-1}\big(s+m(r+1)\big)^{2k-m-1}_{r+1}}=0,
\ee we have
\bes
{[\mathcal{X}^k_r,b_{s}E^{s}_{s}]+[\mathds{X}^k_r, \mathbb{F}_r]=0.}
\ees Note that the equality \eqref{Unpvd} is verified by {\sc Maple}.
\epr

For ease of notation we suppress the indices of some new notations in what follows in this section.
The following expression, denoted by \(\mathfrak{F},\) is needed for the next lemma:
\bas
&&\!\!\!\!\sum^{k+1}_{m=0}\sum^{k-m}_{l=0}\dfrac{b_{s}{a_{r}}^{m+l}(s)^2_2(k+2-m)(k+2)^{m}_{-2}\big((k-2m)(r+1)-s\big)^{l}_{-2(r+1)}}
{2^{m+l+1}m!\big(m(r+1)+s\big)\big(r(m+l)+k+s+2\big)\big((r+1)(m+1)+s\big)^{l}_{r+1}}
E^{(r+1)(m+l)+l}_{(m+l)r+k+s}\\
&&\!\!\!\!+\sum^{k+1}_{m=0}\sum^{k-m}_{l=0}\dfrac{b_{s}{a_{r}}^{m}(s)^2_2(k+2-m)\big((k+2)(r+1)\big)^{m}_{-2(r+1)}g_{l}}{2^{m}m!(r+1)^{m}(k+mr+s+2)}
F^{(r+1)(m+l)+l}_{(m+l)r+k+s}\\
&&\!\!\!\!+\sum^{k+1}_{m=0}\sum^{k-m}_{l=0}\dfrac{b_{s}{a_{r}}^{m+l}(s)_2^2(k+2-m)(k+2)^{m}_{-2}\big((2m-k)(r+1)+s-r\big)^{l}_{2(r+1)}}
{(-1)^{l}2^{m+l+1}m!(k+mr+s+2)(m(r+1)+s+1)^{l+1}_{r+1}}F^{(r+1)(m+l)+l}_{(m+l)r+k+s}.
\eas
\Blem\label{Lemm3.7} Let \(k\) be an odd number and
\bas
\mathds{F}^{-1}_k&:=&\dfrac{a_{r}}{k+1} \mathcal{F}^{-1}_k+\dfrac{{a_{r}}^{k+2}k(r+1)(k+2)^{k+1}_{-2}}{2^{k+1}(k+1)!}\mathcal{Z}^{k(r+1)+r}_{k(r+1)+r},
\eas
where \(\mathcal{Z}^{k(r+1)+r}_{k(r+1)+r}\) is given by equation \eqref{Zcal}. Then,
\be
\left(\mathds{F}^{-1}_k; \dfrac{{a_{r}}^{k+2}k(r+1)(k+2)^{k+1}_{-2}}{2^{k+1}(k+1)!}Z^{k(r+1)+r}_{k(r+1)+r}\right)\in \ker d^{k(r+1)+2r, r+1}.
\ee In addition,
\bes
d^{k(r+1)+s+r, s+1}\left(\mathds{F}^{-1}_k, \0, \mathfrak{F}; \0, \dfrac{{a_{r}}^{k+2}k(r+1)(k+2)^{k+1}_{-2}}{2^{k+1}(k+1)!}Z^{k(r+1)+r}_{k(r+1)+r}\right)=0.
\ees
\Elem
\bpr The proof for the first part is complete by equation \eqref{Zmm} and \eqref{Zcal} together with \cite[Equation 3.4]{GazorMokhtariInt}, that is,
\bas
\left[\mathcal{F}^{-1}_k, \Bbb{F}_r\right]=-\dfrac{{a_{r}}^{k+2}k(r+1)(k+2)^{k+1}_{-2}}{2^{k+1}(k+1)!}F^{k(r+1)+2r}_{k(r+1)+2r}.
\eas
By \cite[Lemma 4.2]{GazorMokhtari}, we have
\bas
&&\left[\mathcal{F}^{-1}_{k}, b_s{E}^s_s\right]+[\mathfrak{F},\Bbb{F}_{r}]
\\&=& \sum^{k+1}_{m=0}\dfrac{b_{s}{a_{r}}^{k+1}(s)^{2}_{2}(k+2-m)(k+2)^{m}_{-2}
\big((k-2m)(r+1)-s\big)^{k-m}_{-2(r+1)}}{2^{k+1}m!\big((k+1)(r+1)+s+1\big)(m(r+1)+s)^{k-m-1}_{r+1}}{E^{k(r+1)+s+r}_{k(r+1)+s+r}}
\eas
Then, the rest of the proof follows the identity
\ba\label{EqMaple}
0&=&\frac{b_{s}{a_{r}}^{k+2}k(r+1)(k+2)^{k+1}_{-2}(s+2)\big((k+1)(r+1)+s\big) }{2^{k+1}(k+1)!
 \big((k+1)(r+1)+s+1\big)(r+1)(k+1)}\\\nonumber
&&+\sum^{k+1}_{m=0}\dfrac{ b_{s}{a_{r}}^{k+2}(s)^{2}_{2}(k+2-m)(k+2)^{m}_{-2}
\big((k-2m)(r+1)-s\big)^{k-m}_{-2(r+1)}}{2^{k+1}m!(k+1)\big((k+1)(r+1)+s+1\big)(m(r+1)+s)^{k-m-1}_{r+1}}.
\ea The equality \eqref{EqMaple} is verified by {\sc Maple}.
\epr

\Blem \label{Lemm3.8} Let
\bas
\mathds{E}^{0}_{2k}&:=&\left(\dfrac{a_{r}(r+1)}{2kr+2k+1}\right)\mathcal{E}^{0}_{2k,r}
+{a_{r}}^{2k+1}e_{2k,2k}(2k(r+1)-r) \mathcal{Z}^{2k(r+1)}_{2k(r+1)},
\eas for any natural number \(k.\) Then,
\be\label{39}
\left(\mathds{E}^{0}_{2k}; {a_{r}}^{2k+1}e_{2k,2k}\big(2k(r+1)-r\big)  Z^{2k(r+1)}_{2k(r+1)}\right)\in \ker d^{2k(r+1)+r, r+1}.
\ee
 In addition, there exists a \(\delta\)-homogenous transformation generator \(\mathfrak{E}\in \LST_{2k+s}\) such that
\ba\label{E2k}
&&d^{2k(r+1)+s, s+1}\left(\mathds{E}^{0}_{2k}, \0, \mathfrak{E}; {a_{r}}^{2k+1}e_{2k,2k}(2k(r+1)-r) Z^{2k(r+1)}_{2k(r+1)}, \0\right)=\\\nonumber
&&\Bigg(\sum^{2k}_{m=0}\dfrac{-b_{s}(k+1)\binom{k}{m}(mr+2k-s)^2_{2(s+1)}\big(2(k-m-1)(r+1)+r-s\big)^{2k-m-1}_{-2(r+1)}}
{(r+1)^{-1}{a_{r}}^{-2k-1}2^{2k-m-1}\big(mr+2(k+1)\big)\big(m(r+1)+s+1\big)^{2k-m-1}_{r+1}}\\\nonumber
&&\;\;+ b_{s}{a_{r}}^{{2k+1}}e_{2k,2k}\frac{(s+2)\big(2k(r+1)-r\big)^2_{r+s+1}}{\big(2k(r+1)+1\big)_{s+1}^2}\Bigg)E^{2k(r+1)+s}_{2k(r+1)+s}
\ea
\Elem
\bpr We recall \cite[Equation 4.9]{GazorMokhtari}
\bas
\left[\mathcal{E}^{0}_{2k,r}, \Bbb{F}_r\right]={a_{r}}^{2k+1}e_{2k,2k}(2k(r+1)-r) F^{2k(r+1)+r}_{2k(r+1)+r}.
\eas Thus, Equations \eqref{Zmm} and \eqref{Zcal} imply equation \eqref{39}.
By \cite[Lemma 4.2]{GazorMokhtari}, for any natural number \(k\) there exists a vector field \(\mathfrak{E}\in \LST_{2k+2}\) so that
\bas
&&[\mathcal{E}^{0}_{2k,r}, b_s{E}^s_s]+[\mathfrak{E},\Bbb{F}_{r}]=\\\nonumber
&&\sum^{2k}_{m=0}\dfrac{{b_{s}}{a_{r}}^{2k}(k+1)\binom{k}{m}(2k+mr-s)^2_{2s+2}\big(2(k-m-1)(r+1)+r-s\big)^{2k-m-1}_{-2(r+1)}}
{2^{2k-m-1}\big(mr+2(k+1)\big)\big(m(r+1)+s+1\big)^{2k-m-1}_{r+1}} E^{2k(r+1)+s}_{2k(r+1)+s}.
\eas
Then, the rest of the proof is straightforward.
\epr
\begin{thm}\label{ONFs+1level}
The \((s+1)\)-th level orbital normal form of \eqref{First2ndlevel} is
\be\label{s+1level}
v^{(s+1)}:=\Theta^0_0+ F^{-1}_0+ \delta F^r_r+\sum^{\infty}_{k=s}\beta_{k}E^{k}_{k}+\sum^{r}_{k=1}\gamma_{k}\Theta^{k}_{k}.
\ee
Here \(\delta:= {\rm sign}(a_r),\) \(\beta_{k}=0\) when \(k\equiv_{2(r+1)}-1,\) and \(k\equiv_{2(r+1)}s\) for \(k>s\). Furthermore, the \(s+1\)-jet of \(v^{(s+1)}\) gives rise to the simplest \(s+1\)-jet orbital normal form.
\end{thm}

\begin{proof} Using changes of variables given above \cite[Lemma 4.1]{GazorMokhtari}, we may change the coefficient \(a_r\) to \({\rm sign}(a_r).\) Thus, we may assume that \(a_r=\delta=\pm1.\) Lemma \ref{Lemm3.8} (see equation \eqref{E2k}) concludes that terms of the form \(E^{2k(r+1)+s}_{2k(r+1)+s}\) are simplified in the \(s+1\)-level. This and Theorem \ref{r+1level} imply that the normal form vector field \eqref{s+1level} can be obtained in the \(s+1\)-level normal form. However, we still need to address the remaining vector fields in \(\ker d^{N,r+1}\) and prove that they do not contribute into further simplification of the system in \(s+1\)-level. In this direction, equation \eqref{Eqn3.9} lists the kernel terms while Lemmas \ref{Lemm3.6} and \ref{Lemm3.7} prove our claim.

Since vector space of all \(\Theta\)-terms is a Lie ideal, the grading \eqref{delta1} for \(\Theta\)-terms can be shifted. Therefore, in what follows we prove
that the kernel term generated by \(Z^0_1\) can not be used for further normalizing (simplifying) \(E\)-terms (instead of \(\Theta\)-terms) in the \(s+1\)-level.
Equation
\begin{equation*}
Z^{r+1}_{r+1} \Bbb{F}_{r}+[\mathcal{Z}^{r+1}_{r+1}, \Bbb{F}_{r} ]=\dfrac{a_{r}(r+1)}{r+2}F^{2r+1}_{2r+1},
\end{equation*} equation \eqref{E01} for \(k=0,\) and equation \eqref{Z01} imply that a linear combination (see equation \eqref{Z01F}) of transformation generators \((\mathcal{Z}^{0}_{1, r}; Z^{0}_{1}),\) \(\mathcal{E}^{0}_{1,r},\) and \((\mathcal{Z}^{r+1}_{r+1}; Z^{r+1}_{r+1})\) belong to \(\ker d^{r+s+1,r+1}.\) Therefore, the linear combination of these can potentially be applied to the system for possible normalization (simplification) in the \(s+1\)-level without effecting terms of lower grades. Hence, we consider the spectral effect of either of these on the vector field \(F^{-1}_0+\delta F^r_r+\beta_s E^s_s.\) (Recall that \(\Theta^0_0\) can be simplified
using Lemma \ref{theta00}.) This is achieved by the following formulas:
\begin{equation*}
 Z^{0}_{1} E^{s}_{s}+ \Le[\dfrac{1}{2(s+1)}E^{s+1}_{s+1}, \Bbb{F}_{r}\Ri]= \dfrac{a_{r}r(r+2)}{2(s+1)(s+r+3)}F^{r+s+1}_{r+s+1}-\dfrac{a_{r}(s+3)}{2(r+s+3)}E^{r+s+1}_{r+s+1},
\end{equation*}
and
\bas
& [\mathcal{Z}^{0}_{1, r}, E^{s}_{s}]+\Le[\frac{(s)^2_2}{3(s+1)^2_2}E^{s+1}_{s+1}-\frac{1}{2(s+2)_1^2}F^{s+1}_{s+1}, \Bbb{F}_{r}\Ri]
&\\ =&{\frac{({r}^{2}+3r-{s}^{2}-2s+2)a_{r}}{2(r+s+3)\left(r+2\right)}}E^{r+s+1}_{r+s+1}
-a_{r}\Le(\frac{(r+1)(2r^2+10r+9)}{6(r+2)^2_{s+1}}-\frac{(r)^2_2(s)_2^2}{3(s+1)^2_2(s+r+3)}-\frac{(s-r+1)}{2(s+2)^2_1}\Ri)
F^{r+s+1}_{r+s+1},&
\eas
while
\bas
{\left[\mathcal{E}^{0}_{1,r}, b_{s}E^{s}_{s}\right]+\Le[\dfrac{b_{s}(s-1)}{2(s+1)}E^{s+1}_{s+1}, \Bbb{F}_{r}\Ri]}&=& -\dfrac{3a_{r}b_{s}r}{2(r+s+3)}E^{r+s+1}_{r+s+1} -\frac{a_{r}b_{s}r(2r+3-s)}{2(r+s+3)(s+1)}F^{r+s+1}_{r+s+1},
\eas \({Z^{r+1}_{r+1} E^{s}_{s}}=E^{r+s+1}_{r+s+1},\) and
\bas
[\mathcal{Z}^{r+1}_{r+1}, E^{s}_{s}] &=& -\dfrac{r+1}{(r+2)^2_{s+1}}F^{r+s+1}_{r+s+1}+
\Le(\dfrac{s-r-1}{(r+3)}+\dfrac{(s)_2^2}{(r+2)_1^2(r+s+3)}\Ri)E^{r+s+1}_{r+s+1}.
\eas Hence,
\ba\nonumber
&&d^{r+s+1,s+1}\left(\mathcal{Z}^{r+1}_{r+1}+\mathcal{Z}^{0}_{1, r}+\mathcal{E}^{0}_{1,r}, \0, \Le(\dfrac{sb_{s}}{2(s+1)}+\dfrac{(s)^2_2}{3(s+1)^2_2}\Ri)E^{s+1}_{s+1}-\dfrac{1}{2(s+2)_1^2}F^{s+1}_{s+1}; Z^{r+1}_{r+1}+Z^{0}_{1}, \0\right)\\
&&\label{Z01F}= -\dfrac{b_{s}(r+1)}{s+2} F^{r+s+1}_{r+s+1},
\ea due to the equality
\begin{equation*}
\frac{r(r-s)-(s+2)^{2}}{(r+2)^2_{s+1}}-\frac{r}{s+r+3}+\frac{(s+2)_r^2}{(r+2)^2_{s+1}}=0.
\end{equation*} Equation \eqref{Z01F} infers that \(F^{r+s+1}_{r+s+1}\) can be simplified. However, equation \eqref{Zmm} indicates that \(F^{m+r}_{m+r}\)-terms for each \(m>0\) have already been simplified in the \(r+1\)-level and no further simplification of \(E\)-terms is possible at this stage.
\end{proof}

\section{ Universal asymptotic unfolding normal form }\label{secPNF}

The conventional approach for local bifurcation analysis of singular parametric differential systems is first to fold the system by setting the parameters to zero and then find the normal form of the folded system. Next, the normalized system is unfolded by adding extra parameter depending terms such that the (versal) unfolded normal form system contains qualitative properties (invariant under an equivalence relation) associated with any small perturbation of the original system; \eg see \cite{YuNonlinearity}. In fact the unfolding also accommodates all possible modeling imperfections. However, this approach has two major disadvantages. Firstly, this approach does not provide the actual relations between the original parameters of a parametric system and the unfolding parameters. This effectively prevents its implementation to bifurcation control. Secondly, most singular systems do not have universal unfolding since many qualitative dynamics of systems can not be determined by any finite jet; see \cite[Page 685]{Murd09} and \cite{BroerHopfZero}. The later explains the reason why we use the notion of universal {\it asymptotic} unfolding normal form. Further, we use a parametric orbital normal form computation in order to compute the parameter relations. This section is devoted to treat Hopf-zero singularities (whose the first few dominant terms are solenoidal) with any possible additional nonlinear-degeneracies. Here, the \(n\)-equivalence relation (\(n\)-jet determined) is used for introducing universal asymptotic unfolding normal form, whose original ideas are due to \cite{MurdBook,Murd09,Murd98} and is amenable to finite normal form computations.

Consider the parametric differential equation
\be\label{PEq1}
\dot{x}:= f(x, y, z, \mu), \; \dot{y}:=z+g(x, y, z, \mu), \; \dot{z}:=-y+h(x, y, z, \mu),\qquad (x, y, z)\in \mathbb{R}^3, \mu\in \mathbb{R}^p.
\ee Here, \(f, g\) and \(h\) are nonlinear formal functions in terms of \((x, y, z, \mu)\) and also they are nonlinear in terms of \((x, y, z)\) when they are evaluated at \(\mu=0\). Remark that the results presented in this paper can be easily generalized to smooth cases using Borel--Ritt theorem
\cite[Theorem A.3.2]{Murd09}.
Equation \eqref{PEq1} represents a multiple parametric perturbation of equation \eqref{Eq1}. Through a sequence of primary shift of coordinates (shifts in \(y\) and \(z\)-variables), we may assume that \(g(0, 0, 0, \mu)= h(0, 0, 0, \mu)= 0\) for all \(\mu\in \mathbb{R}^p\); see the primary and secondary shift of coordinates on \cite[Page 373]{MurdBook} and \cite{Murd98}. Next, it is easy to observe that the vector space spanned by all \(F^i_j, E^i_j, \Theta^i_j\) is the same as the vector space spanned by all resonant vector fields; see \cite[Page 54]{ChowBook}. Therefore, a normal form of equation \eqref{PEq1} can be chosen as \ba\nonumber
v^{(1)}&:=& \sum c_{00\m}\Theta^0_0\mu^\m+\sum a_{-1,-1\m}F^{-1}_{-1}\mu^\m+ \sum a_{-1,0\m} F^{-1}_0\mu^\m+ \sum a_{ij\m}F^i_j\mu^\m
\\\label{1PNF}&&+ \sum b_{ij\m}E^i_j\mu^\m+ \sum c_{ij\m}\Theta^i_j\mu^\m,
\ea where \(c_{00\0}=1\) and \(a_{-1,-1\0}=0.\) It is known that the coefficients given in equation \eqref{1PNF} are not unique and further simplification of \eqref{1PNF} is possible. Using a parametric version of Lemma \ref{theta00}, we may omit \(\sum c_{00\m}\Theta^0_0\) from the system.
Now define the grading function by
\be\label{Pdelta} \delta(F^l_k\mu^\m)=\delta(E^l_k\mu^\m)=\delta(\Theta^l_k\mu^\m)=k+2|\m|.\ee
By similar comments following \cite[Lemma 3.1]{GazorMokhtari} and assuming that \(a_{-1,0\0}\neq 0,\) we can modify \(a_{-1,0\0}\) into \(1.\) Since \({[F^0_{0},F^{-1}_0]}=-2F^{-1}_{0}\), we may simplify all \(F^{-1}_0\mu^\m\)-terms for nonzero \(\m.\) Then, the formulas given in the proof of \cite[Lemma 3.1]{GazorMokhtari} imply that the vector field can be transformed into
\be\label{2ndlevelPar}
v^{(2)}:= F^{-1}_0+ \sum_{i\geq -1} a_{i\m}F^i_i\mu^\m+ \sum_{i\geq 0} b_{i\m}E^i_i\mu^\m+ \sum_{i>0}c_{i\m}\Theta^i_i\mu^\m,
\ee denoting \(a_{-1\m}\) for \(a_{-1,-1\m}\). Define \(r, s\) by
\be\label{Prs} r:=\min \{i\,|\, a_{i\0}\neq0\}\quad \hbox{ and } \quad s:=\min \{i\,|\, b_{i\0}\neq0\}.\ee
We assume that
\be\label{Prls} r<s<\infty.\ee
For further simplification, we apply a new grading structure (compare with equation \eqref{delta1}) generated by
\ba\label{Pardelta1}
\delta(F^{l}_{k}\mu^\m)=\delta(E^{l}_{k}\mu^\m)=r(k-l)+k+(r+1)|\m|, \; \delta(\Theta^{l}_{k}\mu^\m)=r(k-l)+k+s+(r+1)|\m|.
\ea
This grading facilitates the use of results from non-parametric orbital normal forms (\ie Theorem \ref{r+1level}) for parametric cases. Given
\be\label{Z00}
Z^{0}_{0}\Bbb{F}_{r}+ {\left[\dfrac{1}{2}F^0_0, \Bbb{F}_{r}\right]}= \dfrac{a_{r}(r+2)}{2} F^{r}_{r},
\ee and a sequence of secondary shifts in \(x\)-variable using the equation
\be
F^{r}_{r}(x+c(\mu),\rho)= F^{r}_{r}(x,\rho)+ (r+1)c(\mu)F^{r-1}_{r-1}(x,\rho)+ \mathcal{O}(\mu^2)
\ee where \(c(\0)=0,\) we may transform equation \eqref{2ndlevelPar} into the \((r+1)\)-th level parametric normal form
\be\label{r+1levelPar}
v^{(r+1)}:= F^{-1}_0+ a_r F^r_r+ \sum_{-1\leq i<r-1} a_{i\m}F^i_i\mu^\m+ \sum_{0\leq i<s, } b_{i\m}E^i_i\mu^\m+\sum^{\infty}_{k=s}\beta_{k\m}E^{k}_{k}\mu^\m+\sum^{r}_{k=1}\gamma_{k\m}\Theta^{k}_{k}\mu^\m,
\ee for \(a_r:={\rm sign}(a_{r\0}).\) Further, denote \(\beta_s:=\beta_{s\0}\).
Here, \(\beta_{k\m}=0\) for \(k\equiv_{2(r+1)}-1\) and any nonnegative integer-valued vector \(\m.\) Equation \eqref{Z00},
\bes[E^0_0, \Bbb{F}_{r}]= rF^r_r, [F^0_0, E^s_s]= sE^s_s \quad \hbox{ and } \quad [E^0_0, E^s_s]= sE^s_s\ees imply that \(E^s_s\mu^\m\in {\rm im }\,d^{s+2|\m|, s+1}\) when \(\m\neq\0\) and \(s-r \neq 0.\)

For our convenience we define
\bes
N:=r+s-\left\lfloor\frac{s}{2(r+1)}\right\rfloor,
\ees
and a sequence of natural numbers by
\bes
\left\{k_i\,|\, i\in \mathbb{N}, i>r \right\}:= \left\{k\,\Big|\, k\neq s, k\geq 0,\hbox{ and } \frac{k+1}{2(r+1)}, \frac{k-s}{2(r+1)}\not\in \mathbb{N}\right\}.
\ees
Now we are ready to state one of the main results of this paper.
\begin{thm}[Universal unfolding]\label{MainThm}
Assume that the condition \eqref{Prls} holds.  Then,
\begin{itemize}
  \item[I.] there exist an infinite sequence of formal parametric functions \(\nu_i(\mu_j)\) and the finite sequence of formal functions \(\omega_i(\mu_j)\) (for \(1\leq i\leq r\)) such that equation \eqref{PEq1} is equivalent to
\ba\nonumber
v^{(s+1)}&:=& \Theta^0_0+ F^{-1}_0+ a_r F^r_r+ \beta_{s}E^{s}_{s}+ \sum_{1\leq i\leq r}\nu_{i}F^{i-2}_{i-2}+\sum^{N}_{i=r+1}\nu_{i}E^{k_i}_{k_i}\\
&&\label{Univsl}+\sum^{\infty}_{i=N+1}(\beta_{k_i}+\nu_{i})E^{k_i}_{k_i}+\sum^{r}_{i=1}(\gamma_{i}+\omega_i)\Theta^{i}_{i}.
\ea
  \item[II.] The differential system
\ba\label{FinPNF}
\dot{x}&=& 2\rho^2+ a_r x^{r+1}+\beta_sx^{s+1}+\sum_{1\leq i\leq r}\nu_{i}x^{i-1}+\sum_{i=r+1}^{N}\nu_{i}x^{k_{i}+1}, \\\nonumber
\dot{\rho}&=& -\dfrac{ a_r(r+1)}{2}x^r\rho+\frac{1}{2}\beta_sx^{s}\rho-\sum_{1\leq i\leq r}\frac{(i-1)}{2}\nu_{i}x^{i-2}\rho+\frac{1}{2}
\sum_{i=r+1}^{N}\nu_ix^{k_i}\rho,
\\\nonumber
\dot{\theta}&=&1+\sum^{r}_{i=1}(\gamma_i+\omega_{i})x^{i},
\ea is a \(s+1\)-universal asymptotic unfolding normal form for the differential system \eqref{PEq1}.
\end{itemize}
\end{thm}
\bpr
Given the parametric normal form equation \eqref{r+1levelPar},
the proof readily follows by deriving a parametric version of formulas in Theorem \ref{ONFs+1level}. The uniqueness of each polynomial map \(\nu_i(\mu_j)\) and \(\omega_i(\mu_j)\) follows from the uniqueness of the \(s+1\)-jet of orbital normal form.
\epr

\section{ Bifurcation analysis and finite determinacy  }\label{SecBF}

In this section, we prove that the \(2\) and \(3\)-equivalence relations are compatible with the contact equivalence relation for the cases of
\(r:=1\) and \(2\), respectively. Therefore, the steady-state bifurcations and stabilities of equilibria associated with the amplitude system are \(2\) and
\(3\)-determined, respectively. When \(r:=1,\) the secondary Hopf bifurcation is \(3\) and \(5\)-determined for \(s:=2\) and \(3,\) respectively. The proofs follow a systematic approach and well-established theory; \eg see \cite{GolubitskyStewartBook}. We analyze the local bifurcations of equilibria, limit cycles, and secondary Hopf bifurcation of invariant tori of the appropriate finite jet normal form system. For these cases, the hypernormalization up to \(s+1\)-level is necessary and sufficient for analysis and control of these primary and secondary bifurcations. Moreover, the coefficients of the \(s+1\)-level normal form are interestingly sufficient to directly determine the existence and stability type of secondary Hopf bifurcation to invariant tori. The latter indeed only depends on the appropriate unfolding parameter along with \(\beta_2\) and \(\beta_4\) for \(s:=2\) and \(3,\) respectively; see subsections \ref{522} and \ref{523}.

\subsection{Finite determinacy}\label{SubsecFinite}

Finite determinacy or jet-sufficiency is one of the most important challenges in local bifurcation analysis of normal form singular systems. The question is if a finite jet, say \(k,\) of the normal form system has the same qualitative behavior as the original system and no information would be lost by only considering a \(k\)-jet of its normal form. In the affirmative case, we say that the system is \(k\)-determined. In this section we are concerned about the roots of amplitude normal form system. When \(k\)-jet sufficiency is proved, the root bifurcations for a \(k\)-jet of normal form correspond to bifurcations of original system. We use {\it contact-equivalence} in this section, that is, the most natural equivalence relation preserving local roots of smooth maps (germs); see \cite{GazorKazemi} and \cite[Page 166]{GolubitskyStewartBook}.

Finitely determined results for Hopf-zero singularities of codimension two have been reported in the literature using \(C^0\)- and weak \(C^0\)-equivalences; \eg see \cite{TakensSingularities,DumortierSingularitiesR3}. We recall that the three-dimensional Hopf-zero singularity may demonstrate complex dynamical behaviors such as births/deaths of {\it invariant tori}, {\it phase locking}, {\it chaos}, {\it strange attractors}, {\it heteroclinic orbit breakdowns} and {\it \v{S}il'\'{n}ikov bifurcations} which may not be detected by singularity theory and/or normal form methods; see \cite{LangfHopfSteady,LangfTori,LangfHopfHyst,LangfCuspHopf}. In fact we merely address the bifurcation problem of equilibria, limit cycles and secondary Hopf bifurcations of invariant tori; also see \cite{LangfHopfHyst}.

We first claim that the normal form computation of vector fields and application of results from singularity theory are compatible.
It is well-known that for any smooth differential system \eqref{Eq1}, there always exist \(C^\infty\)-smooth changes of coordinates to transform \eqref{Eq1} to a smooth differential system \eqref{First2ndlevel} modulo flat parts; \eg see \cite[Theorem 1]{Murd09}. The transformed equation is further reduced by ignoring the phase component to obtain the {\it amplitude system}. Since further smooth changes of coordinates and time rescaling transform a smooth germ to other contact-equivalent germs, we may instead work with a reduced system obtained from a universal asymptotic unfolding normal form.

Throughout this subsection we follow the notations, terminologies and results of singularity theory from \cite[Chapter XIV]{GolubitskyStewartBook}.
The bifurcations of limit cycles and equilibria are in one-to-one correspondence with those of equilibria for the amplitude system. Hence, we define the map \(F=(F_1, F_2)\) by
\be\label{gr}
F:=\left(\nu_1+ \nu_2 x+2\rho^2+a_r x^{r+1}+\beta_{s}x^{s+1}+\hot, \frac{\nu_2\rho}{2}-a_r x^{r}\rho+\dfrac{\beta_sx^s\rho}{2}+\hot\right).
\ee Here, \(\hot\) stands for smooth maps whose \(s+1\)-jet is zero. The distinguish parameter \(\lambda\) is chosen as \(\nu_2\) for \(r:=1,\) \(\nu_1\) for \(r:=2,\) and remove the remaining parameter by setting to zero.
\subsubsection{The case \(r:=1\). }
We let \(\lambda:=\nu_2\) and \(\nu_1:=0.\) The symmetry group \(\Gamma\) is the trivial (identity) group and thus, it is removed in our notations. Since \(r:=1,\)
\be\label{gr1}
F(x, \rho, \lambda):=\left(\lambda x+2\rho^2+a_1 x^{2}+\beta_{s}x^{s+1}+\hot, \frac{\lambda\rho}{2}-a_1 x\rho+\dfrac{\beta_sx^s\rho}{2}+\hot\right).
\ee
Now we recall some notations from \cite[Chapter XIV]{GolubitskyStewartBook}. The local ring of all {\it smooth scalar germs} in \((x, \rho, \lambda)\)-variables is denoted by \(\E_{x, \rho, \lambda}\). Next we define \(\mathcal{M}\) as the unique maximal ideal of \(\E_{x, \rho, \lambda},\) \ie \(\mathcal{M}:=<x, \rho, \lambda>_{\E_{x, \rho, \lambda}}.\) Further denote \(\overrightarrow{\mathcal{M}}\) for a module over \(\E_{x, \rho, \lambda}\) defined by
\begin{equation}\label{Mcal}
\overrightarrow{\mathcal{M}}:=\bigg<{x\choose 0}, {0\choose x}, {\rho \choose 0}, {0\choose \rho}, {\lambda \choose 0}, {0\choose \lambda }\bigg>_{\E_{x, \rho, \lambda}}.
\end{equation}
Now one can imagine notations like \(\overrightarrow{\mathcal{M}}^3\) and \(\mathcal{M}^3.\) The module \(\overrightarrow{\mathcal{M}}^3\) includes all vector fields whose two-jet is zero; in particular \(\overrightarrow{\mathcal{M}}^3\) includes all flat vector fields. Further, \(\overrightarrow{\mathcal{M}}^3= \mathcal{M}^3{1\choose 0}+ \mathcal{M}^3{0\choose 1}.\)

\begin{lem}\label{Lem5.1} The map (germ) \(F(x, \rho, \lambda)\) is contact equivalent to \(F+p\) for any \(p\in \overrightarrow{\mathcal{M}}^3.\)
\end{lem}
\bpr
We follow \cite[Definition 7.1, Proposition 1.4, Theorem 7.2 and Theorem 7.4]{GolubitskyStewartBook} and instead prove that \(\overrightarrow{\mathcal{M}}^{3}\subseteq\mathcal{K}_s(F).\) The \(\E_{x, \rho, \lambda}\)-module \(\mathcal{K}_s(F)\) is the intrinsic part of the module generated by
\begin{equation} \label{Kappa}\mathcal{M}^2{F_{1x} \choose F_{2x}}, \mathcal{M}^2{F_{1\rho} \choose F_{2\rho}}, \mathcal{M}{F_{1} \choose 0}, \mathcal{M}{F_{2} \choose 0}, \mathcal{M}{0 \choose F_{1}}, \mathcal{M}{0 \choose F_{2}}.
\end{equation}
We choose \(a_1:=1\) to simplify the formulas. For any \(\E_{x, \rho, \lambda}\)-modules \(J\) and \(\mathcal{K}_s,\) the Nakayama's lemma
\cite[Facts 2.4iii, Page 251]{GolubitskyStewartBook} implies that \(J\subseteq \mathcal{K}_s\) if and only if \(J\subseteq \mathcal{K}_s+ \mathcal{M}J.\) Given \(J:=\overrightarrow{\mathcal{M}}^3\) and \(\overrightarrow{\mathcal{M}}^4=\mathcal{M}J,\) we denote \(\cong\) for the equations modulo \(\overrightarrow{\mathcal{M}}^4.\) Since
\bas
&\rho\left(\!\!\begin{array}{c}
F_1\\
0\end{array}\!\!\right)-2x\left(\!\!\begin{array}{c}
F_2\\0
\end{array}\!\!\right)\cong \left(\!\!\begin{array}{c}
3x^{2}\rho+2\rho^3\\0
\end{array}\!\!\right),
x\rho\left(\!\!\begin{array}{c}
F_{1x}\\
F_{2x}\end{array}\!\!\right)-2x\left(\!\!\begin{array}{c}
F_2\\0
\end{array}\!\!\right)
\cong \left(\!\!\begin{array}{c}
4x^2\rho\\-x\rho^2
\end{array}\!\!\right), &
\\&\rho x\left(\!\!\begin{array}{c}
F_{1x}\\F_{2x}
\end{array}\!\!\right)-\rho\left(\!\!\begin{array}{c}
F_1\\0
\end{array}\!\!\right)
\cong \left(\!\!\begin{array}{c}
x^2\rho-2\rho^3\\-x\rho^2
\end{array}\!\!\right),&
\eas
we have \({x^2\rho \choose 0 }, {\rho^3 \choose 0 }, {0 \choose x\rho^2}\in \mathcal{K}_s+\mathcal{M}^4.\) Further,
\begin{equation*}
x{F_2 \choose 0 }\cong{\frac{1}{2}x\lambda\rho-x^2\rho \choose 0 }, \lambda{F_2 \choose 0 }\cong{\frac{1}{2}\lambda^2\rho-x\lambda\rho \choose 0 }, \rho{F_{1\rho} \choose F_{2\rho} }-{0 \choose F_{2} }\cong{4\rho^2 \choose 0 }
\end{equation*} infer that \({x\lambda\rho \choose 0 }, {\lambda^2\rho \choose 0 },\) \({\rho^2\lambda \choose 0 },{\rho^2x\choose 0 } \in \mathcal{K}_s\) modulo \(\mathcal{M}^4\). Next, by
\bes
\rho^2{F_{1x} \choose F_{2x} }, \lambda{0 \choose F_2 }, {0 \choose F_1 }+x{F_{1\rho} \choose F_{2\rho} }\cong {4x\rho \choose 2\rho^2+\frac{3}{2}x\lambda}, x{0 \choose F_2 }, x{ F_{1x}\choose F_{2x} }-2{F_1 \choose 0 }\cong {-x\lambda \choose x\rho }, \rho{0 \choose F_2 },
\end{equation*} we may imply the membership of \({0 \choose \rho^3 },{0 \choose x\lambda\rho }, {0 \choose \lambda^2\rho }, {0 \choose x^2\rho }, {\lambda x^2 \choose 0 }, {0 \choose \lambda\rho^2}\) in \(\mathcal{K}_s+\mathcal{M}^4\). Finally,
\begin{equation*}
2 \lambda{F_1 \choose 0 }-\lambda x{F_{1x} \choose F_{2x} }\cong{x\lambda^2+4\lambda\rho^2 \choose x\lambda\rho }, \lambda^2{F_{1x} \choose F_{2x} }, {0 \choose F_1 }+x{F_{1\rho} \choose F_{2\rho} }\cong {4x\rho \choose 2\rho^2+\frac{3}{2}x\lambda}, \lambda^2{F_{1\rho} \choose F_{2\rho} },
\end{equation*} and \(x{0 \choose F_1 }\) conclude the memberships of \({x\lambda^2 \choose 0 }, {\lambda^3 \choose 0},\) \({0 \choose \lambda x^2 }, {0 \choose x\lambda^2 }, {0 \choose \lambda^3 }\) and \({0 \choose x^3 }\). Now we have proved that all generators of \(\overrightarrow{\mathcal{M}}^3\) belong to \(\mathcal{K}_s(F)\) modulo \(\overrightarrow{\mathcal{M}}^4\). This completes the proof by the Nakayama lemma.
\epr

\begin{rem}
Lemma \ref{Lem5.1} implies that any perturbation \(F+p\) for \(p\in \overrightarrow{\mathcal{M}}^3,\) including the flat perturbations, of the map \(F\) in
\eqref{gr1} is contact-equivalent to \(F.\) Thus, the number of equilibria and limit cycles for the original three dimensional system is \(2\)-determined and is invariant under flat perturbations. This plays an important role in application of normal forms to bifurcation theory due to the Borel lemma, that is, any formal normal form is a smooth normal form modulo a flat function. We recall that subordinate \u{S}il'nikov homoclinic intersections in Hopf-zero singularities is not invariant under flat perturbations and hence, their bifurcations can not be treated by crude uses of normal form methods; see \cite{BroerHopfZero}.

\end{rem}

\begin{cor}
The \(2\)-jet of the map \(F\) in equation \eqref{gr} is a universal unfolding for \(r:=1\) with respect to \(\mathbb{Z}_2\)-contact equivalence relation, where \(\mathbb{Z}_2:=\{\sigma, I\}\) and \(\sigma: \mathbb{R}^2\rightarrow \mathbb{R}^2\) is given by \(\sigma (x, \rho):= (x, -\rho).\)
\end{cor}
\bpr
The proof follows the equivariant universal unfolding theorem \cite[Theorem 2.1 and Equation 2.7, Page 211]{GolubitskyStewartBook}, proof of Lemma \ref{Lem5.1}, and \(T(F)=\overrightarrow{\mathcal{M}}(\mathbb{Z}_2).\)
\epr

\subsubsection{The case \(r:=2\). }

In the following lemma we prove that the equilibria of amplitude normal form system associated with equation \eqref{FinPNF} for \(r:=2\) is \(3\)-determined within the family of \(\mathbb{Z}_2\)-equivariant maps. Here, \(\mathbb{Z}_2\) is generated by \(\sigma: \mathbb{R}^2\rightarrow \mathbb{R}^2,\) \(\sigma (x, \rho):= (x, -\rho).\)
Thereby,
\be\label{gr2}
F(x, \rho, \lambda):=\left(\lambda +2\rho^2+a_2 x^{3}+\beta_{s}x^{s+1}+\hot, -\frac{3}{2}a_2 x^2\rho+\dfrac{\beta_sx^s\rho}{2}+\hot\right).\ee

\begin{lem}\label{LemFinr2} The map \(F\) is \(\mathbb{Z}_2\)-contact equivalent to \(F+p\) for any \(p\in \overrightarrow{\mathcal{M}}^4(\mathbb{Z}_2).\)
\end{lem}
\bpr  The proof closely follows the proof of Lemma \ref{Lem5.1}, the \(\mathbb{Z}_2\)-equivariant and invariant structures, Nakayama's lemma, and the following equalities modulo \(\overrightarrow{\mathcal{M}}^5(\mathbb{Z}_2)\):
\bes
p\left(\!\!\begin{array}{c}
F_1\\
0\end{array}\!\!\right)\cong\left(\!\!\begin{array}{c}
p\lambda\\
0\end{array}\!\!\right), p\left(\!\!\begin{array}{c}
0\\
F_1\end{array}\!\!\right)\cong\left(\!\!\begin{array}{c}
0\\
p\lambda\end{array}\!\!\right),
p\left(\!\!\begin{array}{c}
F_{1\rho}\\
F_{2\rho}\end{array}\!\!\right)\cong\left(\!\!\begin{array}{c}
4 p\rho\\
0\end{array}\!\!\right),
\ees for any monomial \(p\) with \(\deg (p)=3,\) and

\bes
x \left(\!\!\begin{array}{c}
0\\
F_{2}\end{array}\!\!\right)\cong
\left(\!\!\begin{array}{c}
0\\
-\frac{3}{2} x^3 \rho \end{array}\!\!\right),
x^2 \left(\!\!\begin{array}{c}
F_{1x}\\
F_{2x}\end{array}\!\!\right)\cong
\left(\!\!\begin{array}{c}
3 x^4\\
-3 x^3 \rho\end{array}\!\!\right),
\rho^2 \left(\!\!\begin{array}{c}
F_{1x}\\
F_{2x}\end{array}\!\!\right)\cong
\left(\!\!\begin{array}{c}
3 x^2 \rho^2\\
-3 x \rho^3 \end{array}\!\!\right).
\ees
\epr

\subsection{ Bifurcation analysis }\label{BifAnala}

In this section we discuss the local bifurcations of equilibria, limit cycles, torus and a heteroclinic cycle. Here, three of the most generic families (\((r:=1, s:=1),\) \((r:=1, s:=2),\) and \((r:=2, s:=3)\)) are considered; see \cite{LangfTori,LangfHopfSteady,LangfHopfHyst,LangfCuspHopf} and the references therein for the existing literature.

\subsubsection{The case \(r:=1\). }

\begin{figure}
\begin{center}
\subfigure{\includegraphics[width=.5\columnwidth,height=.25\columnwidth]{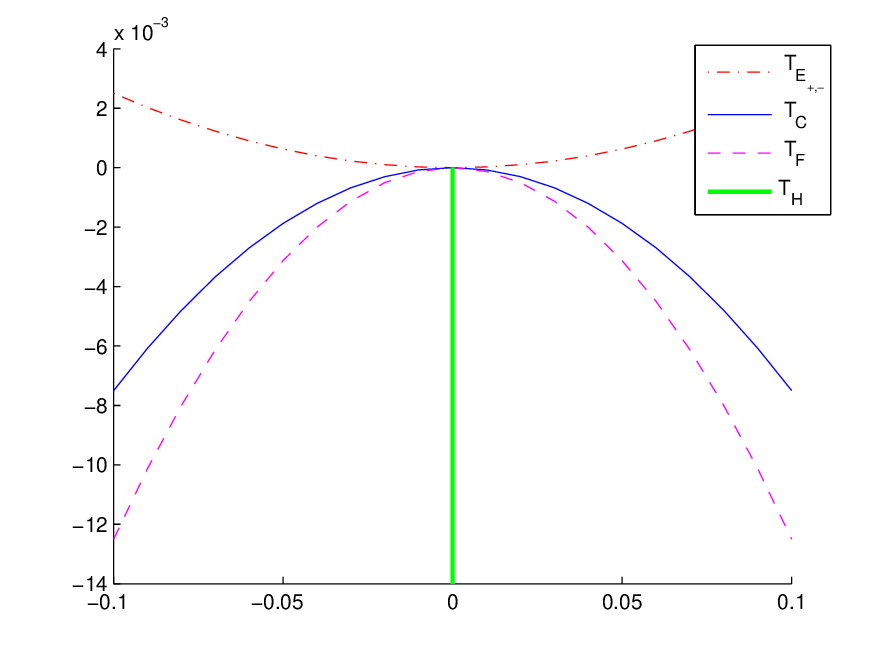}}
\caption{Bifurcation varieties (transition sets) for the system \eqref{EqBF} with \(a_1 =1\): The vertical and horizontal axes stand for \(\nu_1\) and \(\nu_2,\) respectively. }\label{1}
\end{center}
\end{figure}

The amplitude system associated with the two-jet universal asymptotic unfolding normal form \(v\) is given by
\be\label{EqBF}
\dot{x}=\nu_1+2\rho^2+\nu_2x+a_1 x^2, \quad \dot{\rho}=\frac{1}{2}\nu_2\rho-a_1 x\rho.
\ee Note that the \(x\)-axis is always an invariant line. The associated equilibria follow
\bes
E_\pm: (x_{E_\pm},\rho_{E_\pm})=\left(\dfrac{-\nu_2\pm\sqrt{{\nu_2}^2-4a_1 \nu_1}}{2a_1 },0\right)
\hbox{ and } C: (x_C, \rho_C)=\left(\dfrac{1}{2a_1 }\nu_2, \sqrt{-\frac{1}{2}\nu_1-\frac{3}{8a_1 }\nu^2_2}\right).
\ees

The points \(E_\pm\) represent equilibria while \(C\) represents a limit cycle for the three dimensional system.
The transition varieties associated with \(E_\pm\) and \(C\) (depicted in figure \ref{1}) are governed by
\bes T_{E_\pm}:=\Le\{(\nu_1,\nu_2)\mid \nu_1=\dfrac{1}{4a_1 }{\nu_2}^2\Ri\} \hbox{ and } T_{C}:=\left\{(\nu_1,\nu_2)\mid \nu_1=-\dfrac{3}{4a_1 }{\nu_2}^2\right\}.
\ees
Each of these transition varieties \(T_{E_\pm}\) and \(T_{C}\) is associated with a saddle-node bifurcation.

The eigenvalues of the matrices \(Dv(E_{\pm})\) are given by \(\pm\sqrt{{\nu_2}^2-4a_1 \nu_1}\) and \(\nu_2\mp\dfrac{1}{2}\sqrt{{\nu_2}^2-4a_1 \nu_1}.\)
When \((\nu_2<0)\) or \((\nu_2 >0\) and \(4a_1 \nu_1<-3{\nu_2}^2)\) hold, the equilibrium \(E_{+}\) is a saddle point. For \(\nu_2>0,\) \(4a_1 \nu_1>-3{\nu_2}^2,\) \(E_{+}\) is a source. On the other hand, the conditions \(\nu_2>0\) or \(\nu_2 <0\) and \(4a_1 \nu_1<-3{\nu_2}^2\) imply that \(E_{-}\) is a saddle point. However, the conditions \(\nu_2 <0\) and \(4a_1 \nu_1>-3{\nu_2}^2\) conclude that \(E_{-}\) is a sink.

\begin{figure}
\begin{center}
\subfigure[Heteroclinic cycle for \(\nu_2:=0.\)]{\includegraphics[width=.4\columnwidth,height=.23\columnwidth]{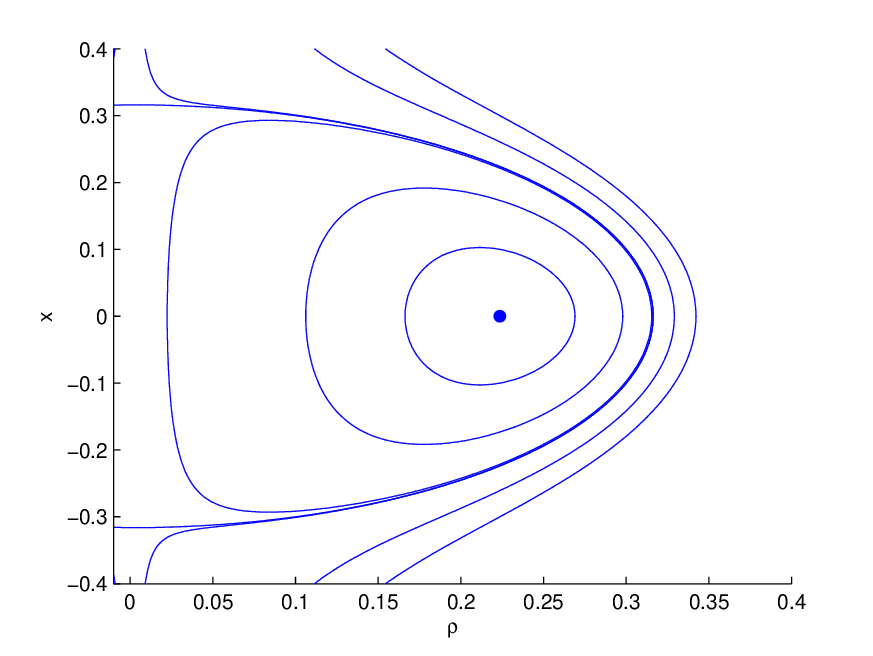}}
\subfigure[Orbits for \(\nu_2:= -0.1.\)\label{Fig2b}]{\includegraphics[width=.4\columnwidth,height=.23\columnwidth]{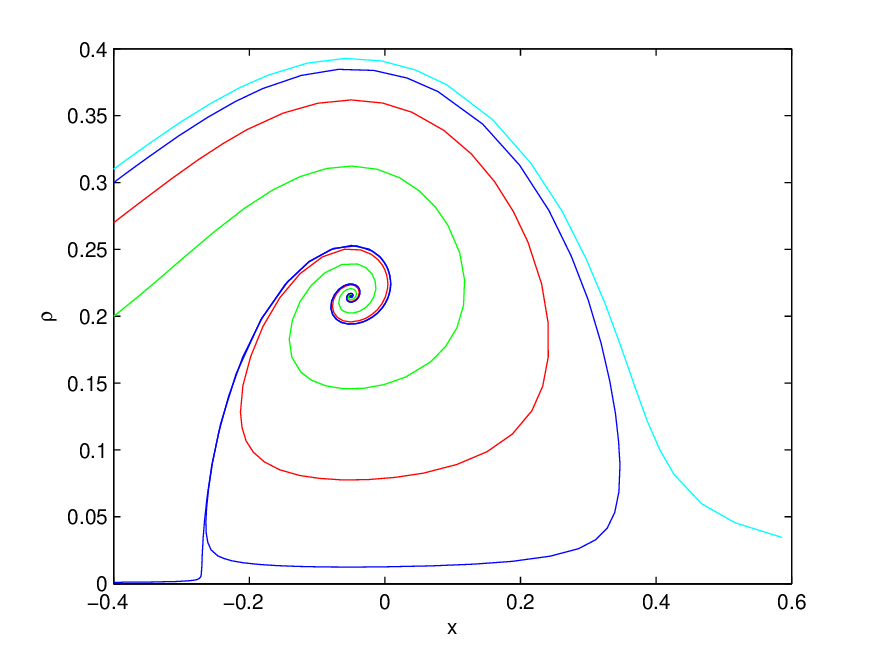}}
\caption{Saddle-saddle connecting cycle surrounding the center point at \(\nu_2:=0\) and its bifurcation to the stable focus point at \(\nu_2:= -0.1\) for the \(2D\) system \eqref{EqBF}, \(a_1:= 1,\) and \(\nu_1:=-0.1.\) }\label{SaddleCon}
\end{center}
\end{figure}
The eigenvalues of \(Dv(C)\) are \(\lambda_\pm=\nu_2\pm\frac{1}{2}\sqrt{8a_1 \nu_1+10{\nu_2}^2}.\) Thus, we define
\be\label{TF}
T_F:= \Le\{(\nu_1,\nu_2)\mid \nu_1=-\dfrac{5}{4a_1 }{\nu_2}^2\Ri\}.
\ee
For \(4a_1 \nu_1<-5{\nu_2}^2,\) \(C\) is a stable/unstable focus point for negative/positve values for \(\nu_2\). If \(-5{\nu_2}^2<4a_1 \nu_1<-3{\nu_2}^2\) holds, for \(\nu_2>0\) the point \(C\) is a source while \(\nu_2 <0 \) concludes that \(C\) is a sink. A pair of pure imaginary eigenvalues occurs at
\bes
T_{H}:=\left\{(\nu_1,\nu_2)\mid \nu_2=0 , a_1\nu_1<0  \right\}.
\ees
\begin{figure}
\begin{center}
\subfigure{\includegraphics[width=.4\columnwidth,height=.23\columnwidth]{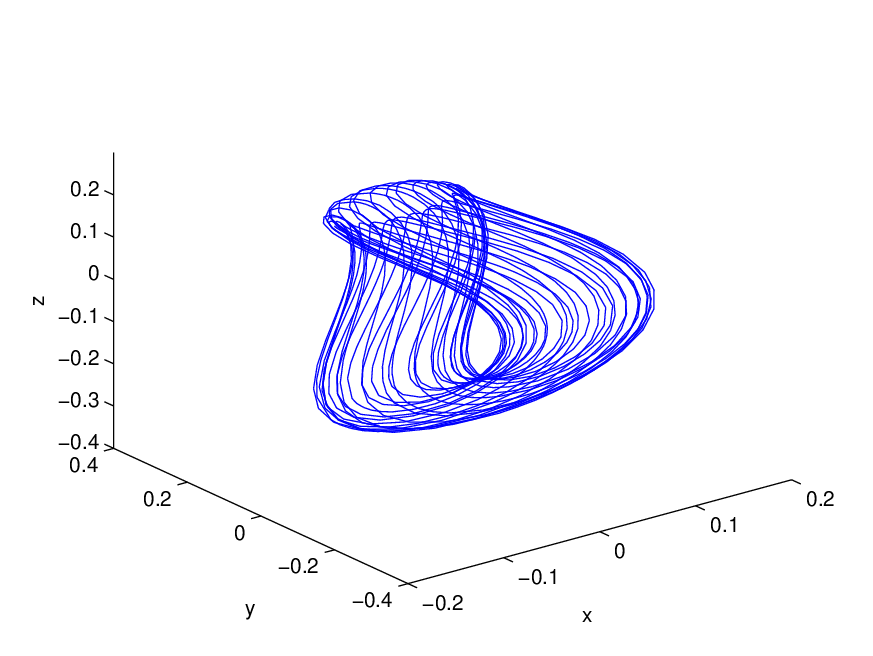}}
\subfigure{\includegraphics[width=.4\columnwidth,height=.23\columnwidth]{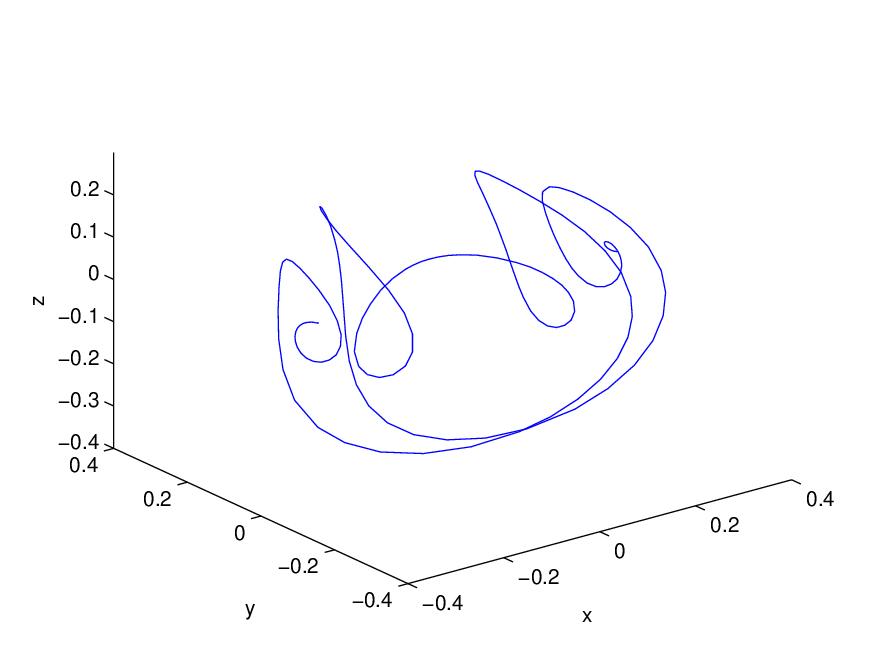}}
\caption{An attracting torus and an estimated saddle-saddle connecting orbit for system \eqref{EqBF} (along with \(\dot{\theta}:=1\) for the \(3D\)-system) at \(\nu_2:=0\), \(a_1:= 1,\) and \(\nu_1:=-0.1.\) Initial values are chosen as \((x, \rho, \theta):= (0.1, 0.2, 0)\) and \((x, \rho, \theta):=(-0.28,0 ,0.02),\) respectively. }\label{SaddleCon3D}
\end{center}
\end{figure}
The system \eqref{EqBF} represents a nonlinear center when the parameters cross the variety \(T_H.\) In fact the system has a first integral given by
\bes
I(x, \rho):= \nu_1\rho^2+ \rho^4+ a_1x^2\rho^2.
\ees This first integral is readily derived given our normal form representation. The saddle equilibria \(E_{+}\) and \(E_-\) belong to saddle-saddle connection cycle surrounding a center in \(2\)-dimensional \((x, \rho)\)-system; see figure \ref{SaddleCon}. The saddle-saddle connection is translated in three dimensional system as an orbit on a \(2\)-sphere surrounding a continuous family of invariant tori around a limit cycle; see figure \ref{SaddleCon3D}. As \(\nu_2\) crosses to positive values, the saddle-saddle heteroclinic cycle breaks; see figure \ref{Fig2b}. However, this heteroclinic cycle is not invariant under \(2\)-equivalence relations and shall not be pursued in our bifurcation control; also see \cite[Page 225]{LangfHopfSteady}.

\subsubsection{The case \(r:=1\) and \(s:=2.\)}\label{522}

Generically, the family of systems for \(r:=1\) at the equilibrium \(C\) undergoes a secondary Hopf bifurcation giving birth to an (attracting) invariant torus for (\(\beta_2<0\) in) the three dimensional system. Furthermore, this secondary bifurcation is three determined for \(\beta_2\neq0\). Consider \(a_1:=1\) in the three-jet system given by
\be\label{EqBFbeta2}
\dot{x}=\nu_1+2\rho^2+\nu_2x+ x^2+\beta_2 x^3+\nu_3x^2, \quad \dot{\rho}=\frac{1}{2}\nu_2\rho-x\rho+ \frac{\beta_2}{2}x^2\rho+\nu_3x\rho,
\ee where \(\beta_2\neq0.\) The qualitative dynamics of equilibria for \eqref{EqBF} and \eqref{EqBFbeta2} are equivalent due to Lemma \ref{Lem5.1}. Thus, we only consider the bifurcation point \(C\) when \(\beta_2\) and \(\nu_3\) are also taken into consideration, whose a Taylor approximation is given by
\bes
(x_C, \rho_C):= \left(\frac{1}{2}\nu_2+\frac{\beta_2}{8}{\nu_2}^2+\frac{1}{4}\nu_2\nu_3, \frac{\sqrt{-2{\beta_2}^2\nu_1-\frac{3}{2}{\beta_2}^2{\nu_2}^2+\frac{3}{4}{\beta_2}^2{\nu_2}^2\nu_3+
\frac{3}{2}\beta_2{\nu_3}^2\nu_2}}{2\,{\rm sign}(\beta_2)\,\beta_2}\right).
\ees
This has a Hopf singularity at \(\nu_2:=0\) and hence,
\be\label{TH}
T_H:=\{(\nu_1, 0, \nu_3)\} \quad \hbox{ and } \quad (x_{C_H}, \rho_{C_H}):= \left(0, \frac{\sqrt{2}}{2}\sqrt{-\nu_1}\right).
\ee By using our {\sc Maple} program \cite{GazorYuSpec,GazorYuFormal}, the three-jet parametric normal form of the system \eqref{EqBFbeta2} in polar coordinates \((\varrho, \vartheta)\) gives rise to
\bes
\dot{\varrho}:= \left(2\beta_2{\nu_2}^2+3\nu_3\nu_2+4\nu_2\right)\varrho+32 \beta_2\varrho^3.
\ees For \(\beta_2>0\) and small values of \(\nu_2<0\) the system experiences an unstable limit cycle while it has a stable limit cycle when \(\beta_2<0\) and \(\nu_2>0\). This is translated to a bifurcation of an attracting invariant torus (for \(\beta_2<0\)) in the original three dimensional system; see figure \ref{AttractingTorus}, where the initial values are chosen as \((x, \rho, \theta):= (-0.57, 0.2, 0).\) It can be seen that \(\nu_1\) and \(\nu_2\) are sufficient to analyze, detect and locate the bifurcations of invariant torus for equation \eqref{EqBFbeta2}. Thus, \(\nu_3\) is omitted in bifurcation controller design.

\begin{figure}
\begin{center}
\subfigure{\includegraphics[width=.4\columnwidth,height=.2\columnwidth]{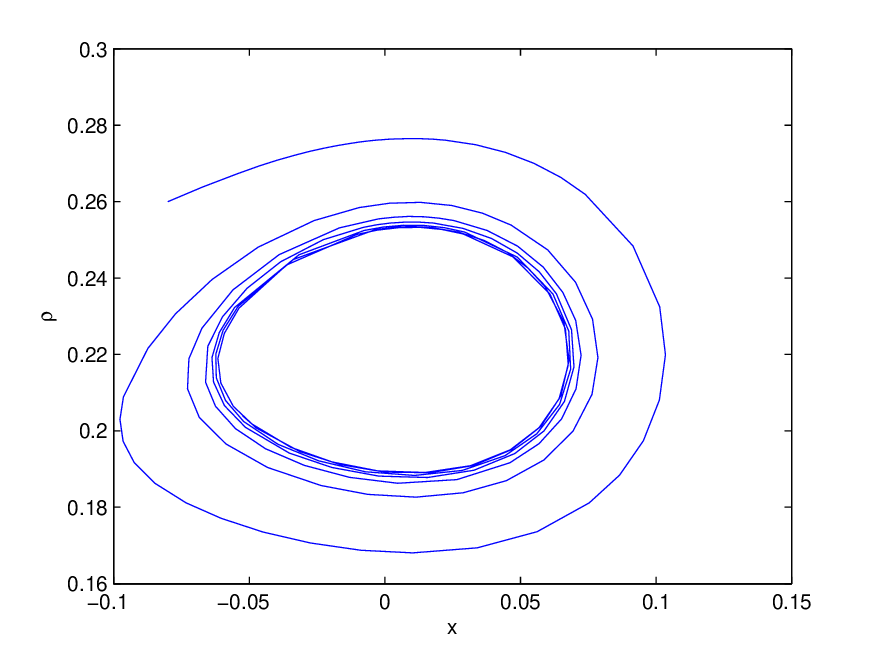}}
\subfigure{\includegraphics[width=.4\columnwidth,height=.2\columnwidth]{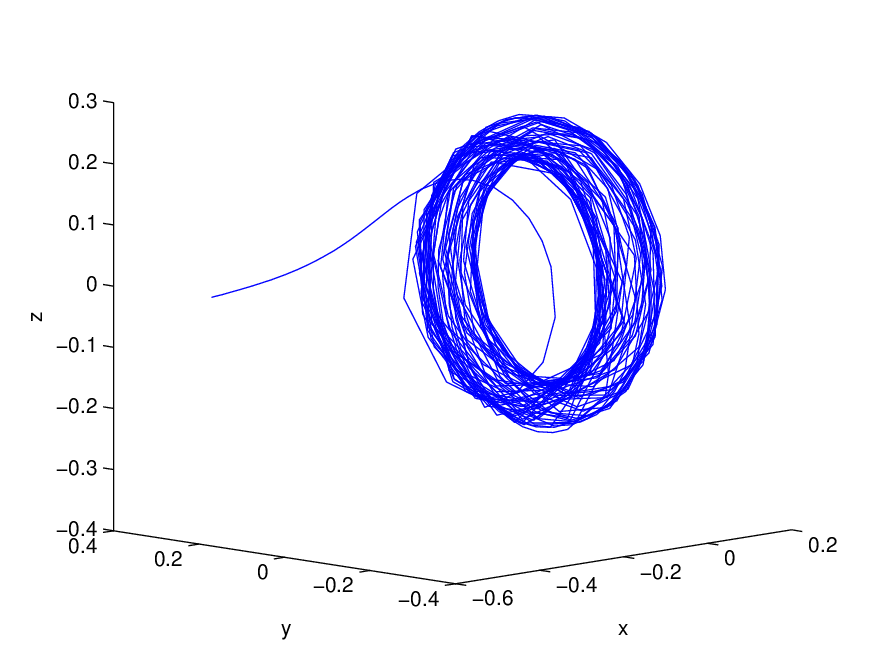}}
\caption{Orbits approaching a stable focus for \(2D\)-system \eqref{EqBFbeta2} and an attracting invariant torus for \(3D\)-system (equation \eqref{EqBFbeta2} along with \(\dot{\theta}:=1\)) at \((\nu_1, \nu_2, \nu_3):=(-0.1, 0.02, 0.1),\) and \(\beta_2:=-10.\) }\label{AttractingTorus}
\end{center}
\end{figure}

\subsubsection{The case \(r:=1\) and \(s:=3.\)}\label{523}

Now we investigate a subordinate bifurcation leading to an invariant torus for a degenerate case where \(a_1:=1,\) \(\beta_{4}\neq0,\) \(r:=1,\) and \(s:=3,\) \ie \(\beta_2:=0, \beta_3\neq0\). This bifurcation is \(5\)-determined and thus, we consider
\begin{eqnarray}\label{2dr1s3}
\dot{x}&=&\nu_{1}+\nu_{2}x+\nu_{{3}}{x}^{2}+{x}^{2}+2{\rho}^{2}+\beta_{3}{x}^{4}+\beta_{4}{x}^{5},\\\nonumber
\dot{\rho}&=&-x\rho+\frac{1}{2}\nu_{2}\rho+\frac{1}{2}\nu_{3}x\rho+\frac{1}{2}\beta_{3}{x}^{3}\rho+\frac{1}{2}\beta_{4}{x}^{4}\rho.
\end{eqnarray} The Hopf singularity occurs at \(\nu_2:=0\), say \(C_H,\) given by equation \eqref{TH} for \(\nu_1<0\). Let \(w:= 2\sqrt{-2\nu_1}\) and treat the bifurcation parameters \((\nu_2, \nu_3)\) as \(||(\nu_2, \nu_3)||=o(|\nu_1|^2).\) By using our Maple program for parametric normal forms of Hopf singularity in polar coordinates \((\varrho, \vartheta)\),
\begin{eqnarray*}
\dot{\varrho}&:=&\left(4\nu_{2}+3\nu_{2}\nu_3+\frac{17}{16}\nu_2{\nu_{3}}^{2}+\frac{33}{32}\beta_{3}{\nu_{2}}^{3}\right)\varrho
\\&&
+\left(99{\frac {\nu_{2}}{{w}^{2}}}+60\beta_3\nu_2+72\beta_4{\nu_{2}}^{2}+\frac{249}{4}\beta_3\nu_2\nu_3+1002\frac{\nu_{2}\nu_{3}}{w^{2}}\right)\varrho^3\\
&&+\left(384\beta_{4}+502{\beta_3}^{2}\nu_{2}+288\nu_3\beta_4+\frac {683617}{12}\frac{\nu_2}{w^4}+{\frac {645385}{24}}{\frac{\beta_{3}\nu_{2}}{w^2}} \right)\varrho^5.
\end{eqnarray*}
Therefore, we have no limit cycle when \(\beta_4\nu_2>0,\) while one stable limit cycle for \(\beta_4<0\) and \(\nu_2>0.\) An unstable limit cycle occurs for \(\nu_2<0\) and \(\beta_4>0.\) Bifurcation of a stable limit cycle here means a secondary bifurcation of attracting invariant torus for the original Hopf-zero system; see figure \ref{AttractingTorusR1S3}.

\begin{figure}
\begin{center}
\subfigure{\includegraphics[width=.4\columnwidth,height=.2\columnwidth]{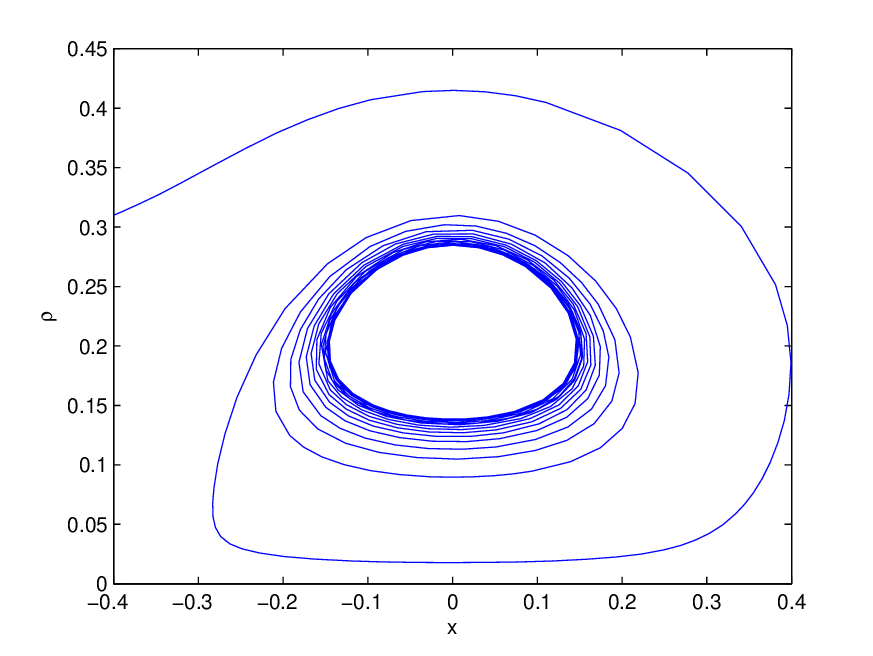}}
\subfigure{\includegraphics[width=.4\columnwidth,height=.2\columnwidth]{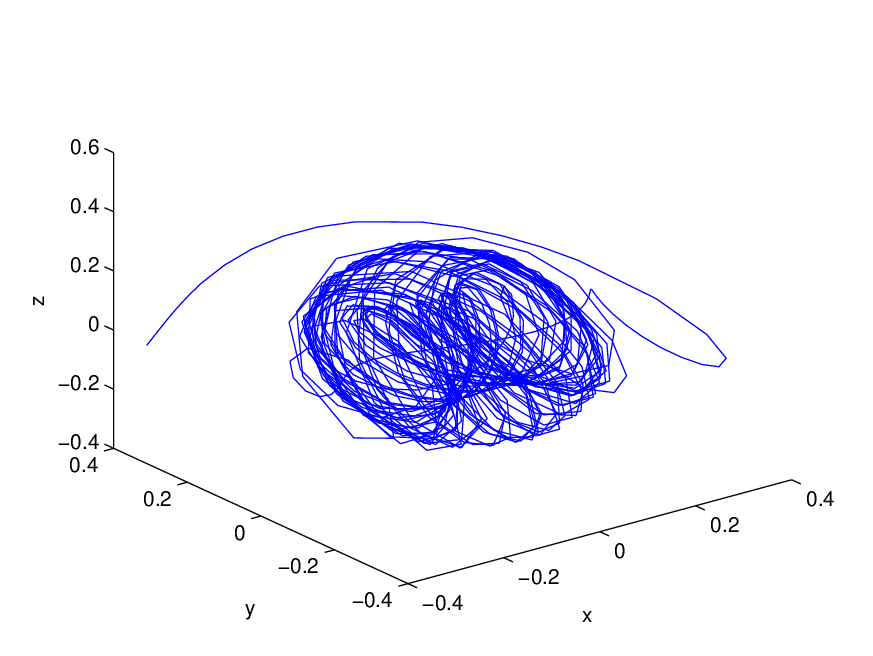}}
\caption{Orbits approaching a stable limit cycle for \(2D\)-system \eqref{2dr1s3} and an attracting invariant torus for \(3D\)-system (equation \eqref{2dr1s3} along with \(\dot{\theta}:=1\)) at \(\nu_1:=-0.1, \nu_2:=0.001, \nu_3:=0.001,\) \(\beta_2:=0, \beta_3:=-1,\) and \(\beta_4:=-10.\) }\label{AttractingTorusR1S3}
\end{center}
\end{figure}

\subsubsection{The case \(r:=2\) and \(s:=3.\)}

This case constitutes a degenerate Hopf-hysteresis singularity; see \cite{LangfHopfSteady,LangfHopfHyst} for a comprehensive literature. For the case of \(r:=2,\) we consider
\begin{eqnarray}\label{Eqr2}
\dot{x}&=&\nu_{{1}}+\nu_{{2}}x+\nu_{{3}}x+2{\rho}^{2}+a_2{x}^{3}+\beta_3x^4,\\\nonumber
\dot{\rho}&=&-\frac{1}{2}\nu_{{2}}\rho+\frac{1}{2}\nu_{{3}}\rho-\frac{3a_2}{2}{x}^{2}\rho+\frac{\beta_3}{2}x^3\rho,
\end{eqnarray} where \(a_2 = \pm1.\) We assume that \(a_2:=1\) and by Lemma \ref{LemFinr2} for bifurcations of equilibria and limit cycles, we merely consider its \(3\)-jet. There are two categories of equilibria for the system \eqref{Eqr2}. The first category is like \(E_i\) given by \((x_i, 0)\) for \(i=1, 2, 3.\) In fact two of these equilibria, say \(E_2\) and \(E_3,\) are born in a saddle-node bifurcation when the parameters cross the variety
\begin{equation*}
T_{E}: \qquad \left\{(\nu_1,\nu_2,\nu_3): 27{\nu_{{1}}}^{2}+4(\nu_{{2}}+\nu_{{3}})^{3}=0\right\}.
\end{equation*} The second category follows
\begin{equation*}
C_\pm: \qquad  \left(x_{C_\pm}, \rho_{C_\pm}\right)=\left(\pm \frac{1}{3}\sqrt{3\nu_3-3\nu_2}, \frac{1}{6}\sqrt {-18\nu_{{1}}\mp 4(\nu_{{2}}+2\nu_3)\sqrt{3\nu_3-3\nu_2}}\right).
\end{equation*} The equilibria \(C_\pm\) represent two limit cycles for the \(3\)-dimensional system (the normal form equation \eqref{Eqr2} along with the phase component). The eigenvalues of \(C_+\) and \(C_-\) are respectively given by
\bes
\nu_{{3}}\pm \frac{\sqrt{3}}{3}\sqrt {11{\nu_{{3}}}^{2}+6\nu_1\sqrt{3\nu_3-3\nu_2}-4\nu_2({\nu_{2}+\nu_3})},
\ees and
\bes\nu_{{3}}\pm \frac{\sqrt{3}}{3}\sqrt {11{\nu_{{3}}}^{2}-6\nu_1\sqrt{3\nu_3-3\nu_2}-4\nu_2({\nu_{2}+\nu_3})}.\ees
\begin{figure}
\begin{center}
\subfigure[\label{Trans2s3a}\(\nu_3:=-0.1\)]{\includegraphics[width=.328\columnwidth,height=.25\columnwidth]{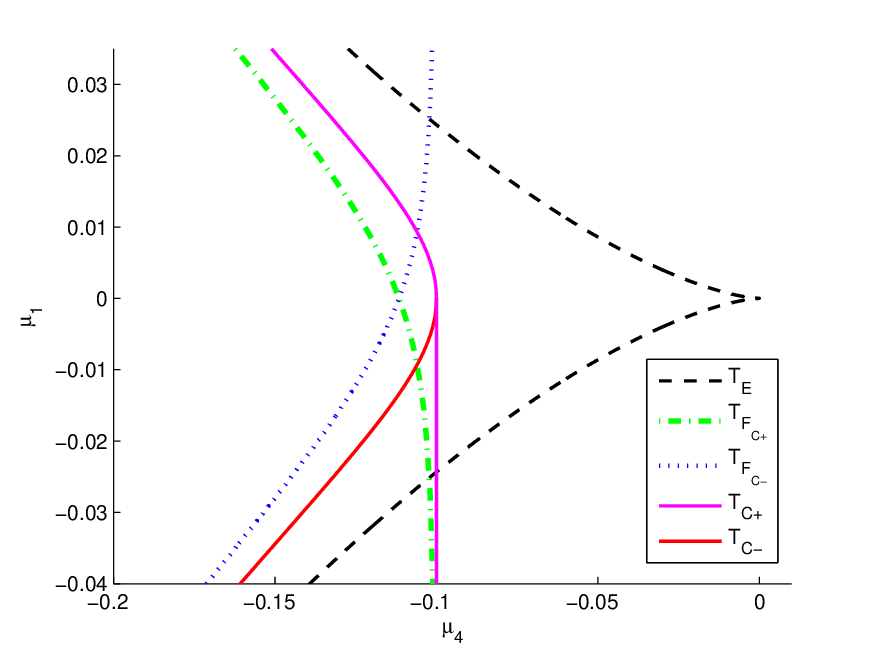}}
\subfigure[\(\nu_3:=0\)]{\includegraphics[width=.328\columnwidth,height=.25\columnwidth]{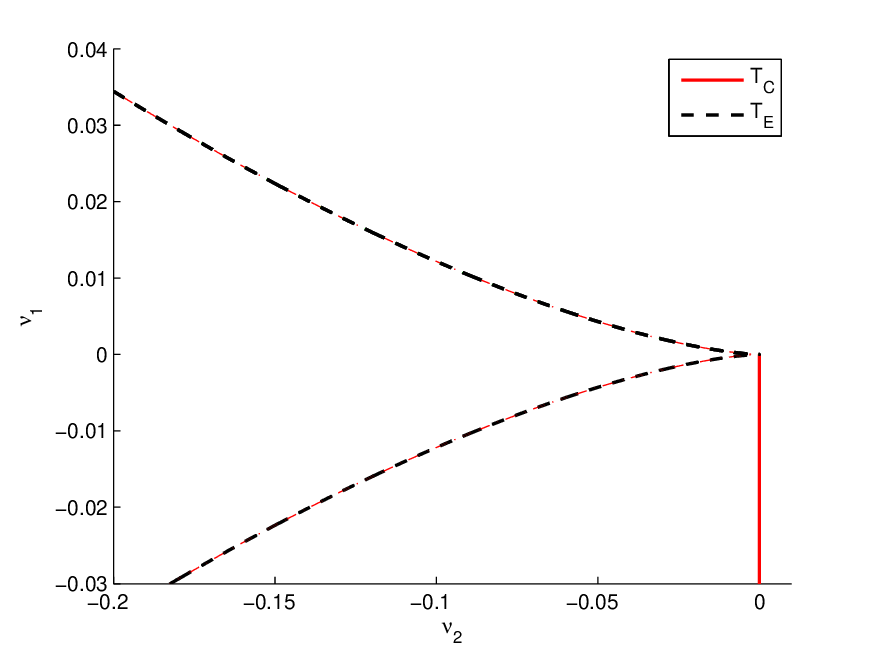}}
\subfigure[\(\nu_3:=0.05\)]{\includegraphics[width=.328\columnwidth,height=.25\columnwidth]{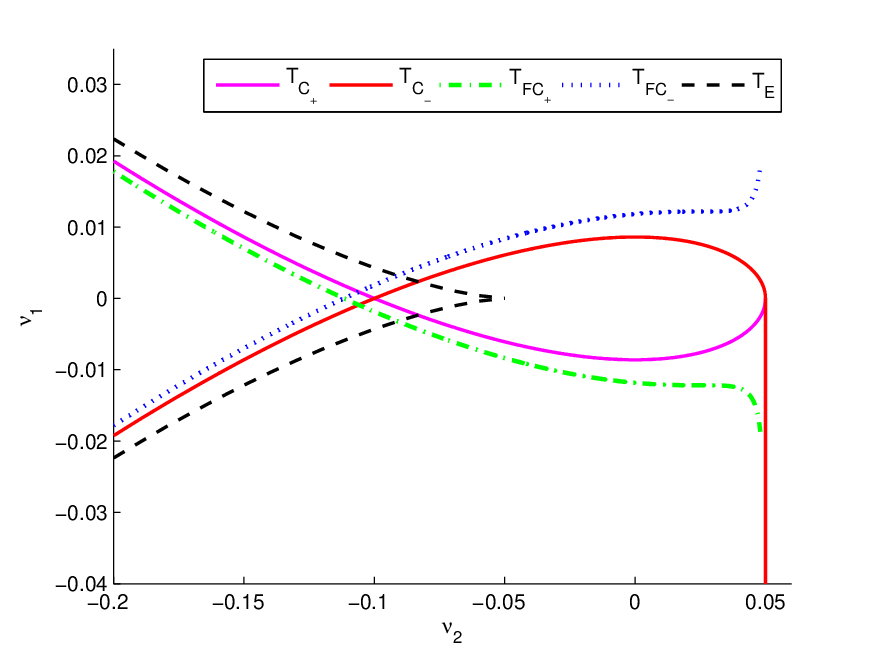}}
\caption{Transition sets \(T_E, T_C,\) and \(T_{FC_{\pm}}\) for \(r:=2, s:=3,\) equation \eqref{Eqr2}. Transition sets of \(T_{FC_{+}}/T_{FC_{-}}\) (and also \(T_{C_+}/T_{C_-}\)) coincide with the upper/lower branch of the cusp \(T_E\) for \(\nu_3:=0.\) There are three equilibria \(E_i\) for \(i:= 1,2,3\) inside the cusp variety \(T_E,\) and one equilibrium \(E_1\) outside the cusp. The points \(C_\pm\) exist/do not exist for parameters below/above the varieties \(T_{C_\pm},\) respectively. For parameters above/below the variety \(T_{FC_+},\) \(C_+\) is a node/focus, while \(C_-\) is always an unstable node. }\label{Trans2s3}
\end{center}
\end{figure}
Thus, two equilibria are born in a saddle-node bifurcation when the parameter crosses
\begin{equation*}
T_{C_\pm}:=\left\{\left(\mp\frac{2\sqrt{3}}{9}(\nu_2+2\nu_3)\sqrt{\nu_3-\nu_2}, \nu_2, \nu_3\right) : \nu_2 \leq \nu_3\right\}\cup \left\{(\nu_1, \nu_2, \nu_2):  \nu_1\leq 0\right\}.
\end{equation*} The equilibria \(C_\pm\) are both unstable when \(\nu_3>0.\) For \(\nu_3<0, \) the equilibrium \(C_+\) is always stable while
\(C_-\) is unstable. Define
\begin{equation*}
T_{F{C_\pm}}:= \left\{(\nu_1,\nu_2,\nu_3): 11{\nu_3}^{2}-4{\nu_{{2}}}^{2}-4\nu_2\nu_3\pm6\nu_1\sqrt{3\nu_3-3\nu_2}=0,\nu_3\neq 0 \right\}.
\end{equation*} The equilibrium \(C_+\) changes from the node type of equilibrium to a focus point when the parameters cross \(T_{FC_+}\); see figure \ref{Trans2s3}, and \cite[figure 6]{LangfCuspHopf} for a similar approach. Since the transition variety \(T_{FC_-}\) is always above \(T_{C_-},\) \(C_-\) is always an unstable node. When the parameter \(\nu_3:=0,\) \(C_+\) has a Hopf singularity while \(C_-\) is a saddle point. Therefore, the system may generically experience a secondary bifurcation at \(C_+.\) This is beyond the scope of this paper and will not be discussed here.

\section{Bifurcation control and universal asymptotic unfolding } \label{SecBC}

Bifurcation control refers to designing a controller for a nonlinear system so that its dynamics gains a desirable behavior; see \cite{ChenBifuControl}. This has many important engineering applications and attracted many researchers. Normal form theory is a powerful tool for local bifurcation control and recently, it has been efficiently used by several authors; see \cite{ChenBifControl2000,ChenBifuControl,Kang04,Kang98,KangIEEE,KangKrener}.
The classical normal form theory is appropriately refined by Kang {\rm el. al.} to include invertible changes of the control state feedbacks; see \cite{Kang04,Kang98,KangIEEE,KangKrener}. Our approach uses hypernormalization (simplification beyond classical normal forms) of the classical normal forms by applying nonlinear time rescaling and also additional nonlinear transformations taken from the symmetry transformation group of the linearized system. This is new in both theory and applications. As far as theory is concerned, this is a contribution to the orbital and parametric normal form classifications of singularities that fits in a long tradition. As a contribution to applications, the extra hypernormalization process enables the use and can propose the type of effective nonlinear state feedback multiple-input controller. The later can be achieved by finding out whether or not a parametric singularity is a universal asymptotic unfolding.

The parametric normal form system \eqref{Univsl} potentially lays the ground for applications in real life problems. Engineering problems are mostly involved with parameters such as control parameters and it is important (when it is feasible) to find explicit direct transformations that transform the asymptotic unfolding parameters to the original parameters of the system. This provides a tool to do the bifurcation analysis of the problem based on the actual controlling parameters. However, this is only feasible when the original system has enough parameters and has them in the right places such that they can actually play the role of the asymptotic unfolding for the system. Therefore, it is important to identify when and which parametric terms of the original system can effectively play the role of unfolding terms and then, remove the redundant parameters. This problem is motivated and greatly influenced by James Murdock and is our most important claimed contribution in this paper. The remainder of this section is devoted to introduce an algorithm for performing this task. This approach potentially proposes certain effective controlling parameters within a parametric system for a possible engineering design. We have implemented our suggested approach in {\sc Maple} to illustrate that it is computable and successfully works.

\begin{rem}
The bifurcation analysis of the universal asymptotic unfolding normal form system provides all possible real world asymptotic (finitely determined) dynamics of an engineering problem. However, the necessity for adding extra unfolding parameters concludes that the original parametric system may not exhibit all such possible dynamics. Hence in these circumstances, the desired dynamics may not always be produced by the existing modeling parameters and modeling refinement is required. Therefore, one needs to find other (already ignored) small parameters in the physics of the problem (our approach provides effective suggestions) to incorporate them in the model for a more comprehensive engineering design.
\end{rem}
Denote the \(s+1\)-jet of the vector field \eqref{Univsl} by
\be\label{FinPNFRed}
\tilde{v}(x, \rho, \nu_1, \cdots, \nu_N):=\begin{pmatrix}
2\rho^2+a_rx^{r+1}+\beta_sx^{s+1}+\sum_{1\leq i\leq r}\nu_{i}x^{i-1}+\sum_{i=r+1}^{N}\nu_{i}x^{k_{i}+1}\\
-\dfrac{a_{r}(r+1)}{2}x^r\rho+\frac{1}{2}\beta_sx^{s}\rho-\sum_{1\leq i\leq r}\frac{i-1}{2}\nu_{i}x^{i-2}\rho+\frac{1}{2}
\sum_{i=r+1}^{N}\nu_ix^{k_i}\rho\\
\end{pmatrix},
\ee and the polynomial map \(\nu(\mu)\) and matrix \(J\) by
\be\label{nuJacb}
\nu(\mu):= (\nu_1(\mu), \ldots, \nu_N(\mu)), \qquad J:= \frac{\partial (\nu_1, \ldots, \nu_N)}{\partial (\mu_1,\ldots, \mu_p)}\bigg|_{\mu=\0}.
\ee Assume that
\bes\rank (J)= k \quad \hbox{ for }\quad k\leq \min \{ p, N \}.\ees
Then, there always exists a linear space \(M\) such that
\ba\nonumber
\mathbb{R}^p= \ker J\oplus M.
\ea Similar to the formal basis style \cite[Page 1006]{GazorYuSpec} in finding complement spaces, we may choose the complement space \(M\) such that
\(M:= \Span\{e_{\sigma(i)}\,|\, i= 1, \ldots, k\}\) for a permutation \(\sigma \in S_{p}\). Here, \(e_j\) denotes the standard basis of \(\mathbb{R}^p.\) Let \(\hat{e}_i:=Je_{\sigma(i)},\) for \(i=1, \ldots, k.\) Thus,
\bes
{\rm range} (J)= \Span \{\hat{e}_i\,|\, i= 1, \ldots, k\}.
\ees Define \(\hat{\mu}:=(\mu_{\sigma(k+1)}, \ldots, \mu_{\sigma(p)})\) and the polynomial map \(\psi_{\hat{\mu}}: M\rightarrow \mathbb{R}^k\) by
\ba
&\psi_{\hat{\mu}}(\mu_{\sigma(1)}, \mu_{\sigma(2)}, \ldots, \mu_{\sigma(k)})= (\nu\cdot \hat{e}_1, \nu\cdot \hat{e}_2, \ldots, \nu\cdot \hat{e}_k), &
\ea where \(\nu\) is the polynomial map given in equation \eqref{nuJacb}. Since the Jacobian of \(\psi_{\0}\) evaluated at the origin has the full rank and
assuming that \(\hat{\mu}\) is sufficiently small, the map \(\psi_{\hat{\mu}}\) is locally invertible. Then,
\be\label{embedtrans}
\left(\mu_{\sigma(1)}(\nu, \hat{\mu}), \mu_{\sigma(2)}(\nu, \hat{\mu}), \ldots, \mu_{\sigma(k)}(\nu, \hat{\mu})\right)={\psi_{\hat{\mu}}}^{-1}\big(y_1(\nu), \ldots, y_k(\nu)\big),
\ee where \(y_i(\nu)=\nu\cdot \hat{e}_i\) for \(i=1, \ldots, k\). Combining the map given by \eqref{embedtrans} and \(\nu\) given by \eqref{nuJacb}, the following proposition holds.
\begin{prop}\label{ParNF} Assume that \(\rank (J)= k\) and consider \(\tilde{v}\) in equation \eqref{FinPNFRed}, the permutation \(\sigma\in S_p\) described above, and invertible reparametrizations \(\mu_{\sigma(i)}(\nu, \hat{\mu})\) (for \(i=1, \ldots, k\)) given by \eqref{embedtrans}.
Then, there exist polynomial functions \(\nu_{\sigma(i)}(\mu)\) (for \(i=k+1, k+2, \ldots, N\)) so that equation \eqref{PEq1} is equivalent to
\be\label{RepPNF}
[\dot{x}, \dot{\rho}]=\tilde{v}\left(x, \rho, \mu_{\sigma(1)}, \mu_{\sigma(2)}, \ldots, \mu_{\sigma(k)}, \nu_{\sigma(k+1)}(\mu), \ldots, \nu_{\sigma(N)}(\mu)\right),
\ee that is, the \(s\)-universal asymptotic unfolding planar normal form.
\end{prop}
\bpr Given Theorem \ref{MainThm} and thanks to equation \eqref{embedtrans}, a straightforward computation shows that \eqref{PEq1} takes the form of \eqref{RepPNF}; also see the proof of \cite[Lemma 3.3]{GazorYuFormal}. \epr
We refer to the parameters \(\mu_{\sigma(1)}, \mu_{\sigma(2)}, \ldots, \mu_{\sigma(k)}\) as {\it distinguished parameters}, \ie they play the role of universal asymptotic unfolding parameters. Adding extra asymptotic unfolding parameters to the system is only justified when the original control parameters of the system can not fully do the unfolding job. In this case, the control parameters may still have influences upon the added unfolding parameters and any such relation is useful for their applications in bifurcation control and needs to be computed. This is computed through Proposition \ref{ParNF}.
The map given by \eqref{embedtrans} projects the transition sets associated with the planar differential system \eqref{FinPNFRed} from the \(\nu\)-variables into the original variables \(\mu_{\sigma(i)}\) for \(i=1, \ldots, k\).

\section{ Illustrating examples with two imaginary uncontrollable modes }\label{secExm}

This paper is concerned with designs of nonlinear controllers for linearly uncontrollable systems; \ie linear uncontrollability means that linear controllers are not sufficient to force an initial state to another state in a finite time and thus, nonlinear controllers are instead designed to control its dynamics. We emphasize that we do not treat control systems but we indeed provide cognitive suggestions for designing controllers for linearly uncontrollable differential systems. Kang and Krener \cite{KangKrener} described an extended Brunovsky canonical form for {\it linearly controllable} {\it nonlinear control systems};
see \cite[Theorems 1 and 3]{KangKrener} and \cite[Equation 2.11a]{KangKrener}. \cite[Equation 2.11a]{KangKrener} accommodates a single zero singularity among its simplest examples while at least two imaginary uncontrollable modes are needed for a linearly uncontrollable system. This justifies Hopf-zero singularity  amongst the simplest examples of a nonlinear system that is not linearly controllable. Hopf bifurcation control is readily available given our results in \cite{GazorYuFormal,GazorYuSpec}. The Bogdanov-Takens bifurcation controller design is an in progress project.

We have developed a {\sc Maple} program to compute the parametric normal forms for any small perturbation of a family of Hopf-zero singularity (the case I).
Further, it may take constant coefficients (different from perturbation parameters) of the parametric systems as unknown symbols rather than merely taking numerical coefficients. This greatly promotes its potential for practical applications. We here appreciate Sajjad Bakrani-Balani for writing a procedure that enhanced the efficiency of this capability. Our program will be updated as our research progresses aiming at integrating and enhancing our results \cite{GazorYuSpec,GazorYuFormal,GazorMokhtariInt,GazorMokhtariEul,GazorMokhtari,GazorMoazeni} into a user-friendly {\sc Maple} library for normal form analysis of singularities.

Any control system \cite[Equations (2.1--2.2)]{Kang04} on a three dimensional central manifold with two imaginary uncontrollable modes can be transformed into
\be\label{EqCont}
\dot{x}:= u+f(x, z_1, z_2, u), \quad \dot{z_1}:=z_2+ g(x, z_1, z_2, u), \quad \dot{z_2}:=-z_1+ h(x, z_1, z_2, u),
\ee using linear changes in state and feedback variables; see \cite[Equation (2.3)]{Kang04}. Here \(u\) stands for a quadratic multiple-feedback controller and assume that it is given by
\be\label{QuadControl}
u:= \mu_1+\mu_2z_2+\mu_3z_1+\mu_4x+\mu_5{z_2}^2+\mu_6{z_1z_2}+\mu_7x{z_2}+\mu_8{xz_1},
\ee where \(\mu_i\)-s stand for the control parameters.
\subsection{The case \(r:=1.\) }

We assume that
\be\label{EqCoef}
f:=-d_1x^2+d_2 {z_1}^2+d_3x^3, \quad g:=d_1x z_1, \quad h:= d_1x z_2, \quad d_1d_2d_3\neq 0.
\ee
\begin{figure}
\begin{center}
\subfigure{\includegraphics[width=.5\columnwidth,height=.25\columnwidth]{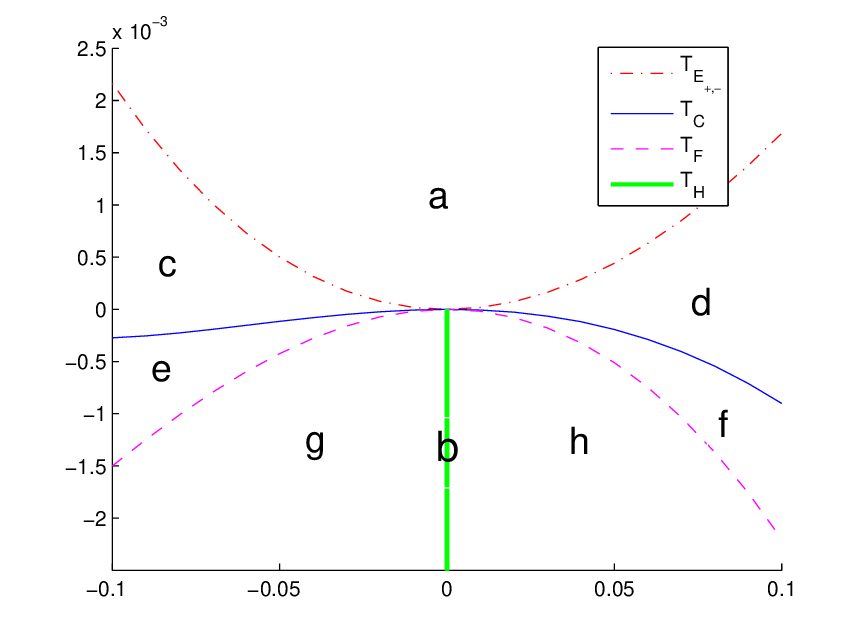}}
\caption{The two-jet estimated varieties in terms of original distinguished parameters \((\mu_1, \mu_4)\) in equation \eqref{EqCont}.
The vertical and horizontal axes are \(\mu_1\) and \(\mu_4,\) respectively. }\label{2}
\end{center}
\end{figure}
Our approach can be applied to similar examples. A universal asymptotic unfolding normal form \eqref{EqBF} is described by
\bes
a_1=- \dfrac{4d_1}{d_2}, \qquad \beta_2= \dfrac{3d_3}{d_2},
\ees and
\bas
\nu_1(\mu)&=&{\frac{4}{d_2}}\mu_1-\frac {8}{{d_2}^2}\mu_{1}\mu_{5}
-\frac{13d_{1}}{{2d_{2}}} {\mu_{1}}^2+
\frac{1}{2d_{1}d_{2}} {{\mu_{4}}^{2}}+
\frac{1}{2}{\frac {{d_{1}}^{2}}{d_{2}d_{3}}}\mu_{1}\mu_{{4}}+
{\frac {1679
}{3456}}{\frac {{d_{3}}^{2}{\mu_{1}}^{2}}{d_{2}{d_{1}}^{3}}}
+{\frac {365}{64}}{\frac {d_
{3}\mu_{1}\mu_{{4}}}{{d_{1}}^{2}d_{2}}},
\\
\nu_2(\mu)&=&{\frac {2}{d_{2}}}\mu_{4}
-\frac{2}{{d_{2}}^{2}}\mu_{4}\mu_{5}
+\frac{{d_{1}}^{2}}{8d_{2}d_{3}} {\mu_{4}}^{2}
-{\frac {2d_{1}}{d_{2}}}\mu_{1}\mu_{4}+
{\frac {1271}{768}}\frac {d_{3}{\mu_{4}}^{2}}{{d_{1}}^{2}d_{{
2}}}+{\frac {15419}{
13824}}{\frac {{d_{3}}^{2}\mu_{1}\mu_{4}}{d_{2}{d_{1}}^{3}
}}.
\eas
\begin{figure}
\begin{center}
\subfigure[Region a: Asymptotic to the \(x\)-axis, the solution approaches infinity for \(\mu_1=0.004, \mu_4=0.001\).]
{\includegraphics[width=.49\columnwidth,height=.18\columnwidth]{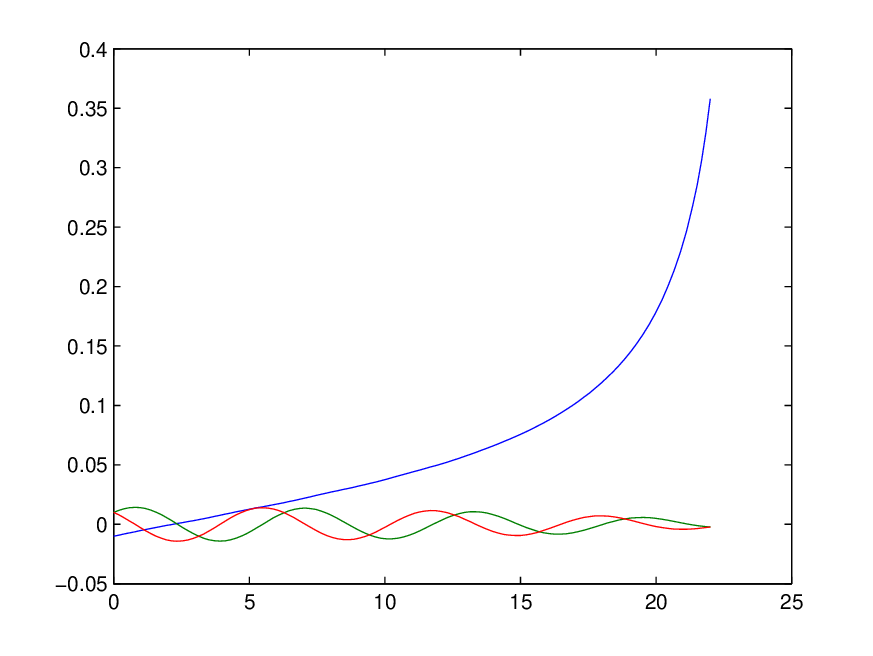}}
\subfigure[Parameters \(\mu_1=-0.004\) and \(\mu_4=0.001\) are taken from region b. Note that here \(d_3:=\frac{3}{4}\)).]
{\includegraphics[width=.49\columnwidth,height=.18\columnwidth]{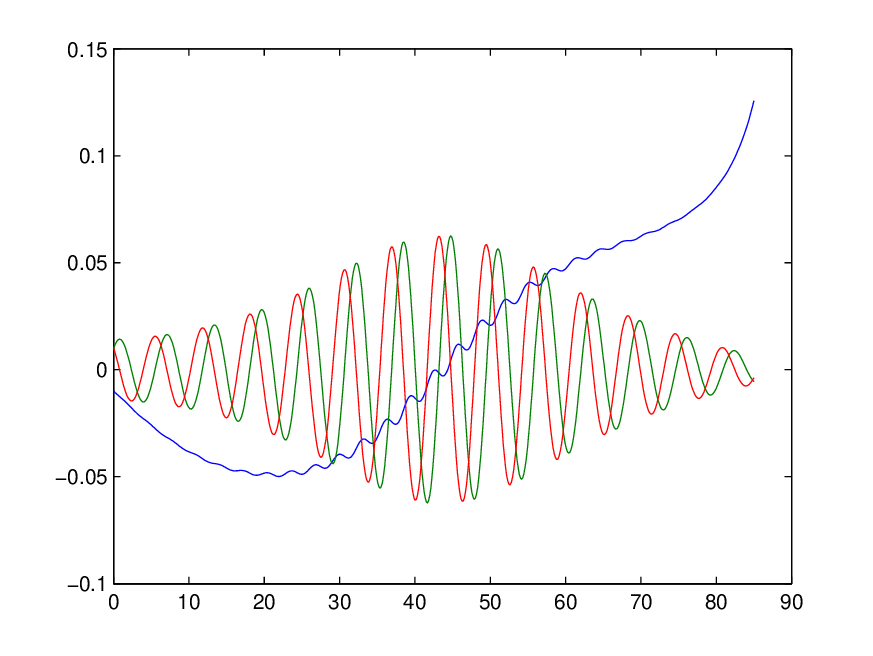}}
\subfigure[Region c: The orbit converges to a stable equilibrium for \(\mu_1=0.0005\), \(\mu_4=-0.1\), \ie \(E_-\) is a sink.]
{\includegraphics[width=.49\columnwidth,height=.18\columnwidth]{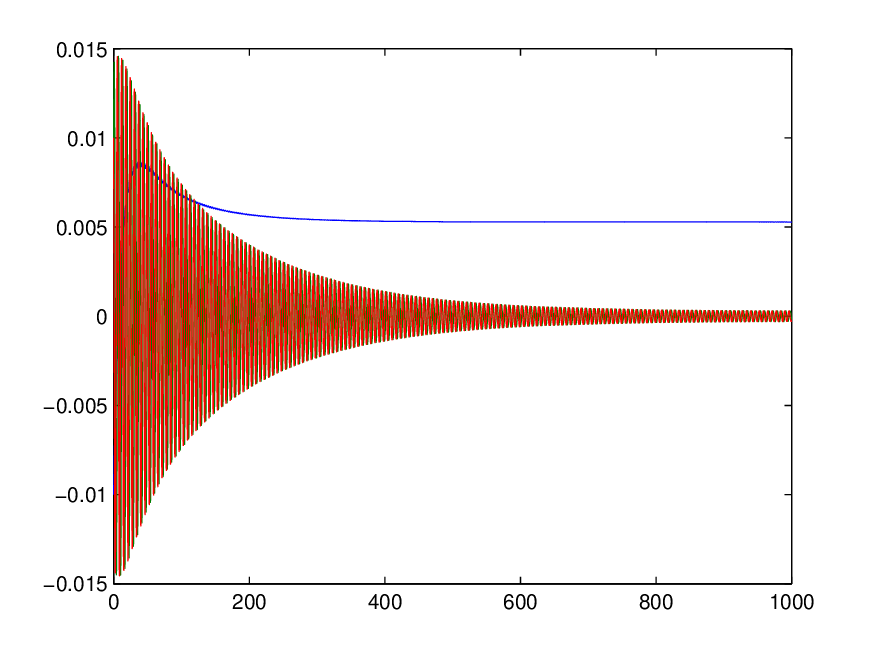}}
\subfigure[Region d: For \(\mu_1=0.0005\), \(\mu_4=0.08,\) the equilibrium \(E_-\) is a saddle while \(E_+\) is a source point. ]
{\includegraphics[width=.49\columnwidth,height=.18\columnwidth]{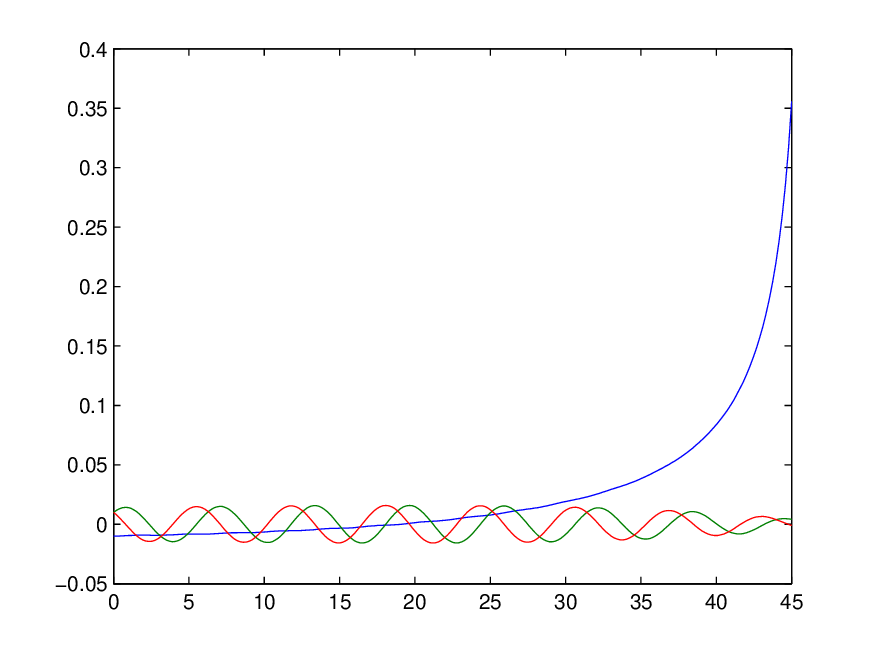}}
\subfigure[The orbit approaches a stable limit cycle (corresponding to \(C\)) in 3-dimension for \(\mu_1=-0.001\), \(\mu_4=-0.08\) from region e.]{\includegraphics[width=.49\columnwidth,height=.18\columnwidth]{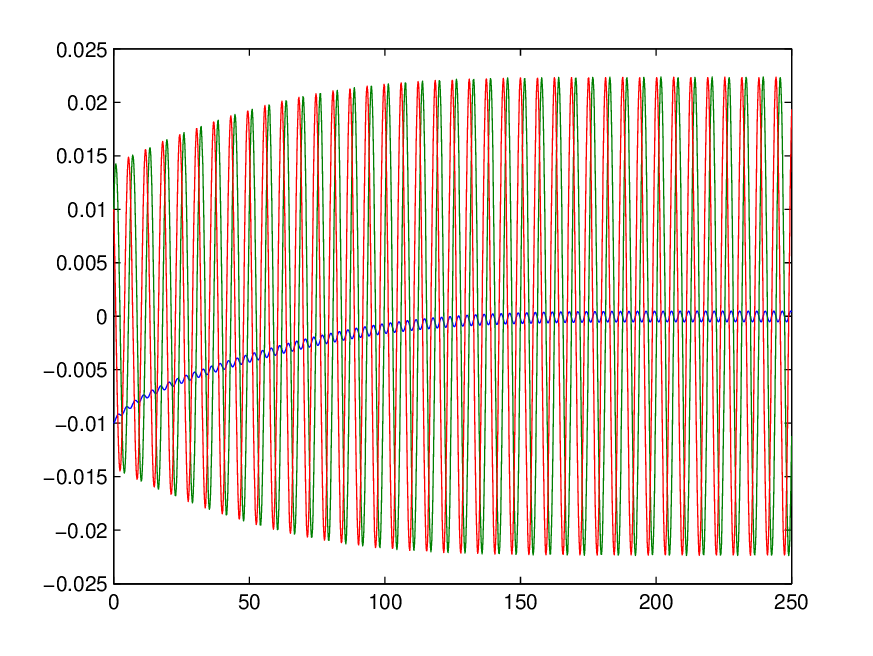}}
\subfigure[The orbit diverges to infinity while rotates around the \(x\)-axis. Here, \(C\) is a source and \(\mu_1=-0.001,\) \(\mu_4=0.08\) from region f]
{\includegraphics[width=.49\columnwidth,height=.18\columnwidth]{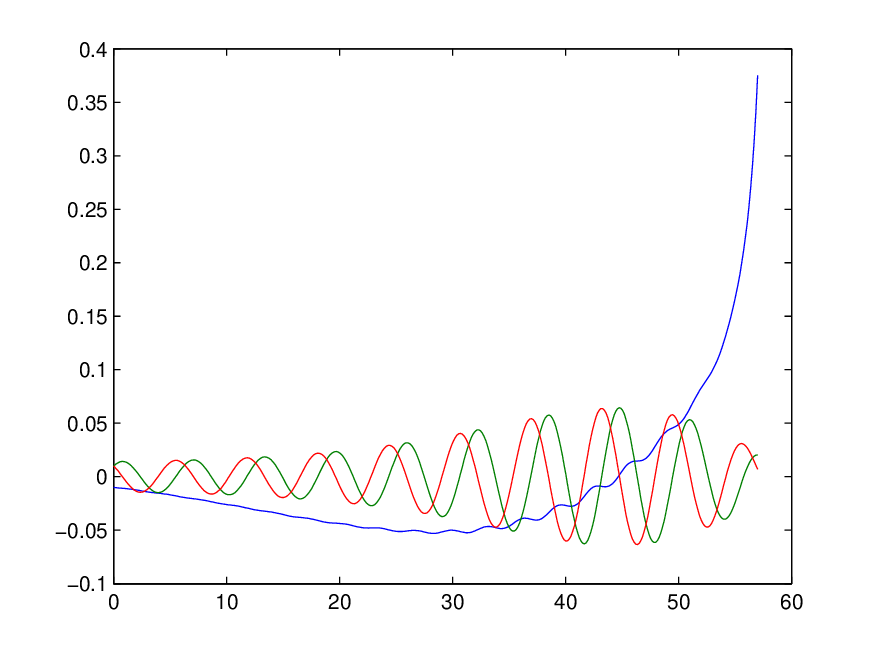}}
\subfigure[Region g: The solution rotates and converges to a stable limit cycle \(C\) for \(\mu_1=-0.0015,\) \(\mu_4=-0.06.\)]
{\includegraphics[width=.49\columnwidth,height=.18\columnwidth]{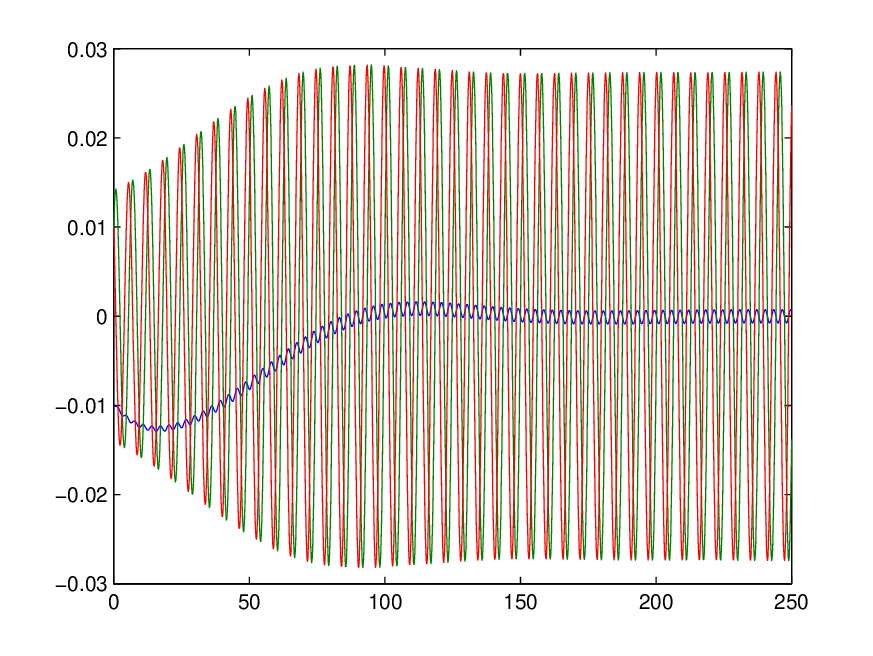}}
\subfigure[Region h: An attracting invariant torus for \(d_3:= -\frac{3}{4}\) and \(\mu_1=-0.0015,\) \(\mu_4=0.001.\) ]
{\includegraphics[width=.49\columnwidth,height=.18\columnwidth]{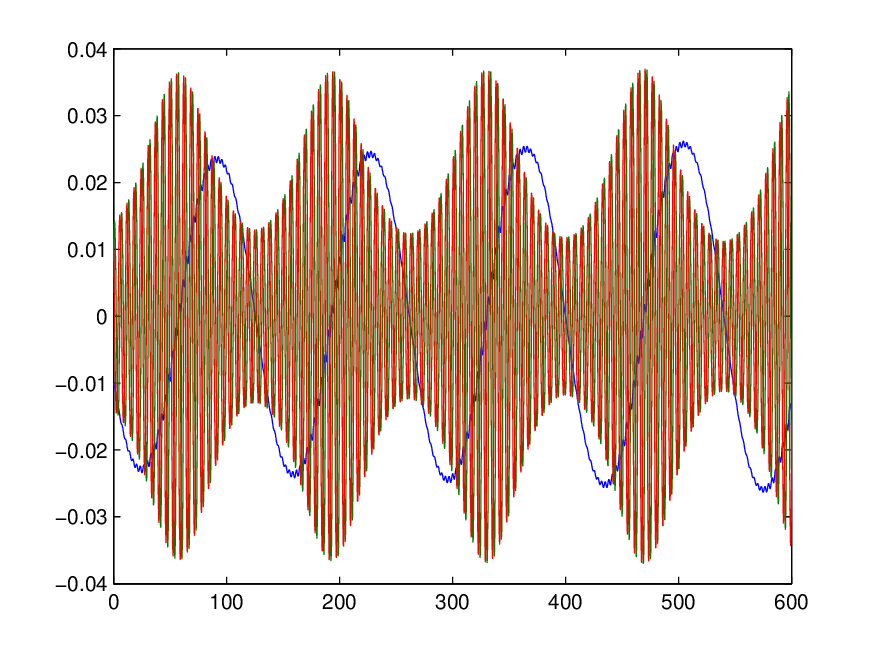}}
\caption{Time series \((x(t), y(t), z(t))\) for different choices of parameters from regions (a--h) in figure \ref{2}. }\label{3}
\end{center}
\end{figure}

Using our {\sc Maple} program, the distinguished parameters are \(\mu_1\) and \(\mu_{4}\). Therefore, we choose \(\mu_j=0\) for any \(j\neq 1, 4.\)
Then, the two-jet approximations for the inverse maps \({\nu_1}^{-1}\) and \({\nu_2}^{-1}\) are governed by
\bas
\mu_1(\nu_1, \nu_2)&:=&\frac{1}{4}d_{2}\nu_{1}+{\frac {13d_{1}{d_{2}}^{2}}{128}}{\nu_{1}}^{2}-\frac{{d_{2}}^{2}}{32d_{1}} {{\nu_{2}}^{2}}
-{\frac {365}{2048}}{\frac {{d_{2}}^{2}d_{{
3}}\nu_{1}\nu_{2}}{{d_{1}}^{2}}}-{\frac {1}{64}}{\frac {{d_{2}}^{2}{d_{1}}^{2}\nu_{1}
\nu_{2}}{d_{3}}}\\
&&-{\frac {1679}{221184}}{\frac {{d_{2}}^{2}{{d_{3}}^{2}\nu_{
1}}^{2}}{{d_{1}}^{3}}},
\\
\mu_{4}(\nu_1, \nu_2)&:=&\frac{d_{2}}{2}\nu_{2}+\frac{d_{1}{d_{2}}^{2}}{8}\nu_{1}\nu_{2}-{\frac {1271}{6144}}{\frac {{d_{2}}^{2}d_{3}{\nu_2}^{2}}{{d_1}^{2}}}-{\frac {{d_{1}}^
{2}{d_{2}}^{2}}{64d_{3}}}{\nu_{2}}^{2}-{\frac {15419}{221184}}{\frac {{d_{2}}^{2}{d_{3}}^{2}\nu_{1}\nu_{2}}{{d_{1}}^{3}}}.
\eas
For numerical simulation, we choose
\bes d_1:= -1, \quad d_2:= 4, \hbox{ and } d_3:= \frac{4}{3}.\ees
Then, \(a_1:=1\) and \(\beta_2:=1.\)
Thus, a three-jet estimation in terms of \(\mu_{1}\) and \(\mu_{4}\) for the bifurcation varieties \(T_{E_\pm}\) and \(T_{C}\) are given by
\bas T_{E_\pm}&:=&\left\{\left(\mu_1, \frac{3}{16}{\mu_{{4}}}^{2}-{\frac {1061}{4608}}{\mu_{{4}}}^{3}+{\frac {
3486907}{7077888}}{\mu_{{4}}}^{4}-{\frac {31594532951}{36691771392}}
{\mu_{{4}}}^{5}+{\frac {29226779813459}{21134460321792}}{\mu_{{4}}}^{6}\right)\right\},
\\
T_C&:=&\left\{\left(\mu_1, -\frac{1}{16}{\mu_{{4}}}^{2}-{\frac {157}{512}}{\mu_{{4}}}^{3}+{\frac {
22826407}{63700992}}{\mu_{{4}}}^{4}-{\frac {347712683}{452984832}}
{\mu_{{4}}}^{5}+{\frac {92683866456661}{63403380965376}}{\mu_{{4}}}^{6}\right)\right\},
\eas
\bes
 T_{F}:=\left\{\left(\mu_1, -\frac{3}{16}{\mu_{{4}}}^{2}-{\frac {1589}{4608}}{\mu_{{4}}}^{3}+{\frac{4797371}{21233664}}{\mu_{{4}}}^{4}
 -{\frac{23112314429}{36691771392}}{\mu_{{4}}}^{5}+{\frac {8519077413067}{7044820107264}}{\mu_{{4}}}^{6}\right)\right\},
\ees and \(T_H:=\{(\mu_1, 0)\}\); see figure \ref{2}. To validate our approximated transition sets, we choose the following eight values from regions (a-h) respectively:
\bas
(\mu_1, \mu_4):&=& (0.004, 0.001), (-0.004, 0.001), (5 \times 10^{-4}, -0.1), (5 \times 10^{-4}, 0.08), \\
 && (-0.001, -0.08), (-0.001, 0.08), (-1.5\times 10^{-3}, -0.06), (-1.5\times 10^{-3}, 0.001).
\eas Numerical simulations using MATLAB and the initial point \((x, y, z)=(-0.01, 0.01, 0.01)\), we accordingly obtain figure \ref{3}(a)--(g). To observe an attracting invariant torus, we instead choose \(d_3:=-\frac{4}{3}\); see figure \ref{3}(h). These are compatible with our anticipated bifurcations and demonstrate that cognitive choices for different values of \((\mu_1, \mu_4)\) from connected components of figure \ref{2}
effectively control the local dynamics of the system \eqref{EqCont}.

\subsection{The case \(r:=2\) and \(s:=3.\)}
\begin{figure}
\begin{center}
\subfigure{\includegraphics[width=.45\columnwidth,height=.27\columnwidth]{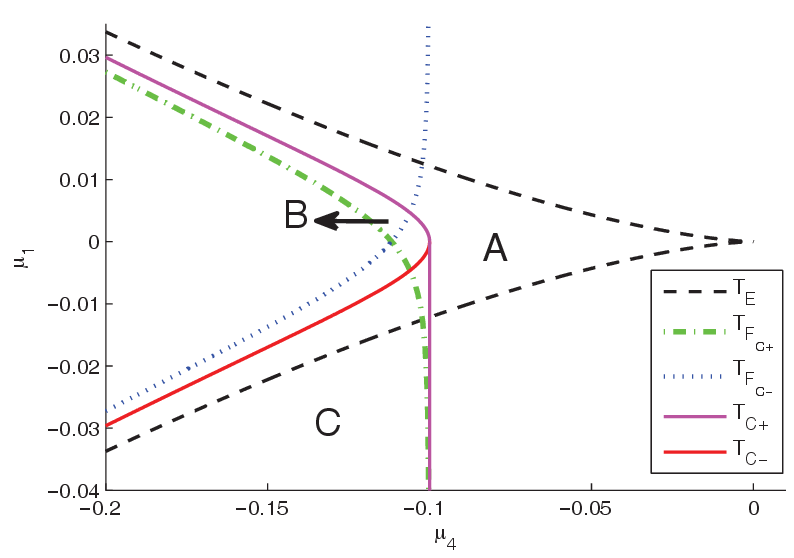}}
\caption{Transition varieties for the system (\ref{EqCont}, \ref{EqCoef}, \ref{EqCoef0}), \((r, s):=(2, 3),\) and \(\mu_3+\mu_4:=-0.1\). \(B\) shows the region above \(T_{FC_+}\) and below \(T_{C_+}.\) }\label{invTrans}
\end{center}
\end{figure}

Consider the control system given by equations \eqref{EqCont}, \eqref{EqCoef} and
\be\label{EqCoef0}
f:= d_2 {z_1}^2+d_3x^3+d_4x^4, \quad g:= u-\frac{3}{2}d_3x^2z_1, \quad h:= -\frac{3}{2}d_3x^2z_2, \qquad d_2d_3d_4\neq 0.
\ee Using our Maple program the distinguished parameters are chosen as \(\mu_1, \mu_3,\) and \(\mu_4,\) and in the parametric normal form \eqref{Eqr2}, \(a_2:= \frac{4d_3}{d_2},\) \(\beta_3:= \frac{16d_4}{5d_2},\)
\bes
(\nu_1,\nu_2,\nu_3):=\left(\dfrac{2 }{d_2} (2\mu_1+\mu_1\mu_3 \mu_4),\frac{2}{d_2}\Le(\mu_4-\mu_3+2 \mu_3 {\mu_4}^2\Ri),\frac{2}{d_2}(\mu_3+\mu_4)\right),
\ees and
\bes
(\mu_1,\mu_2,\mu_3):=\left(\frac{d_2}{4} \nu_1-\frac{{d_2}^3}{128}\nu_1{\nu_4}^2+\frac{{d_2}^3}{128}\nu_1 \nu_3^2, \frac{d_2}{4} \nu_4-\frac{d_2}{4} \nu_3,  \frac{d_2}{4} \nu_4+\frac{d_2}{4} \nu_3\right).
\ees Hence, estimated transition sets are given by
\bes
T_E:=\Le\{(\mu_1, \mu_3, \mu_4): \frac{432{\mu_1}^2(1+ \mu_3 \mu_4+0.25{\mu_3}^2{\mu_4}^2)}{{d_2}^2}+\frac{256{\mu_4}^3(1+3\mu_3{\mu_4}+3{\mu_3}^2{\mu_4}^2+{\mu_3}^3 {\mu_4}^3)}{{d_2}^3}\Ri\},
\ees
\bes
T_{C_\pm}:=\Le\{\Le(\frac{\mp4\sqrt{3\mu_3(1- {\mu_4}^{2})}}{9\sqrt{d_2}}\frac{\mu_3+3\mu_4+2\mu_3{ \mu_4}^{2}}{2+\mu_3 \mu_4}, \mu_3, \mu_4\Ri): d_2\mu_3\geq 0\right\}\cup\{(\mu_1, 0, \mu_4): d_2\mu_1\leq 0\},
\ees

\bes
T_{FC_\pm}:=\Le\{(\mu_1, \mu_3, \mu_4): \sqrt {{\mu_3}}\mu_1=\mp\frac {11{\mu_{3}}^{2}+30\mu_{3}\mu_{4}+3{\mu_{4}
}^{2}-24\mu_3{\mu_4}^{3}+8{\mu_3}^{2}{\mu_4}^{2}-
16{\mu_3}^{2}{\mu_4}^{4}}{6\,{\rm sign}(d_2)\sqrt {3d_2{(1- {\mu_4}^{2})}}(2+\mu_3\mu_4) }
\right\},
\ees
and
\bes
T_H:=\{(\mu_1, \mu_3, -\mu_3)\}.
\ees For numerical simulations we take \(d_2:=4, d_3:= 1, d_4:= \frac{5}{4},\) \(\mu_3+\mu_4:=-0.1,\) and the following values for parameters \((\mu_1, \mu_4)\):
\bes
(-0.0047, -0.112), (-0.005, -0.08),(-0.035, -0.15), (-0.035, -0.15),
\ees
and respectively obtain figures \ref{invTrans}(a)--(d). Figures \ref{C}--\ref{D} suggest the coexistence of two distinct attractors, that is a bistability involving equilibria and limit cycles in region \(C\) of figure \ref{invTrans}.
\begin{figure}
\begin{center}
\subfigure[Region \(A\): An orbit approaching a stable equilibrium \(E_1\).]
{\includegraphics[width=.40\columnwidth,height=.23\columnwidth]{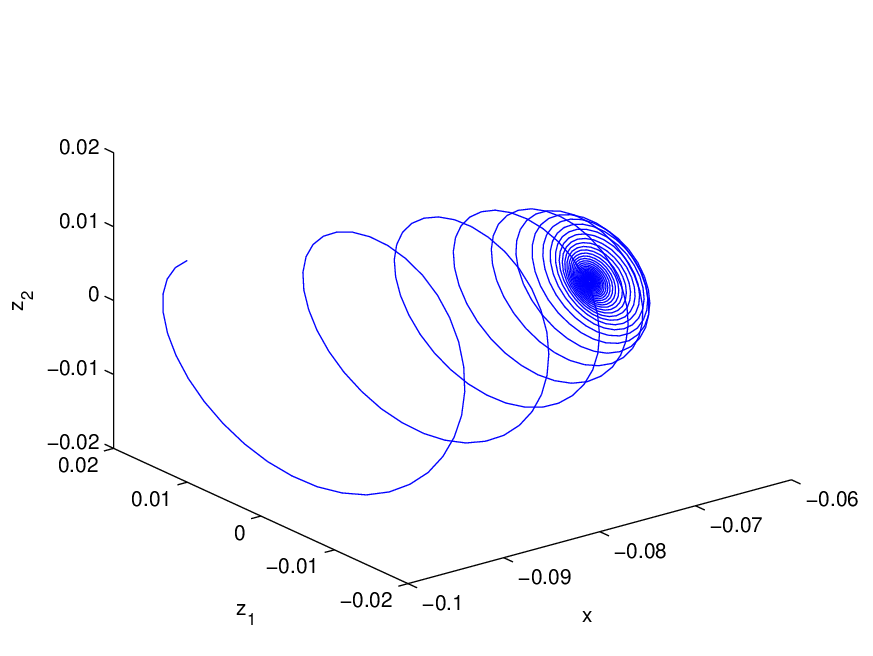}}
\subfigure[Region \(B\): An orbit approaching a stable limit cycle \(C_+\).]
{\includegraphics[width=.40\columnwidth,height=.23\columnwidth]{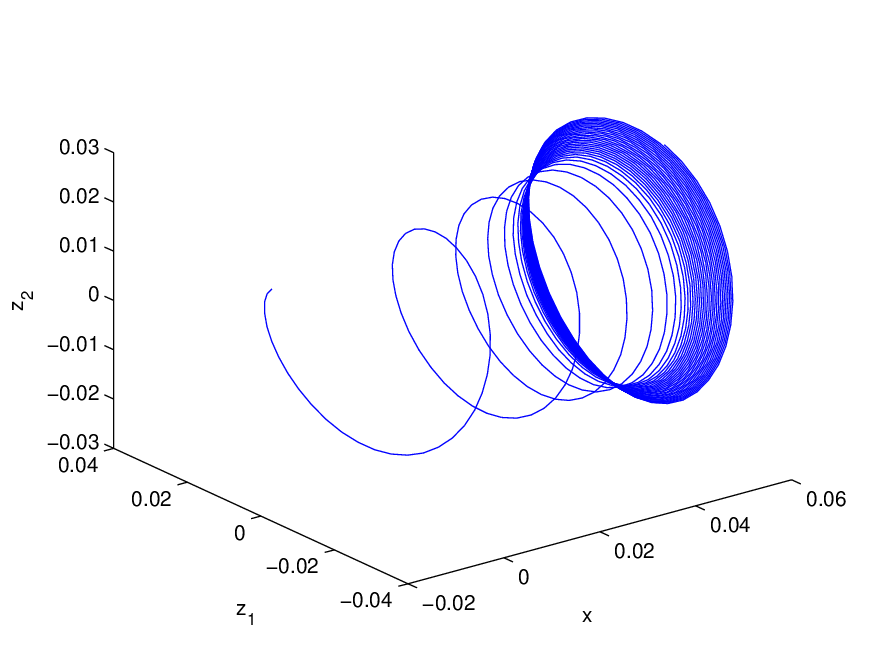}}
\subfigure[\label{C} Region \(C\): An orbit approaching and rotating around the limit cycle \(C_+\).]
{\includegraphics[width=.40\columnwidth,height=.23\columnwidth]{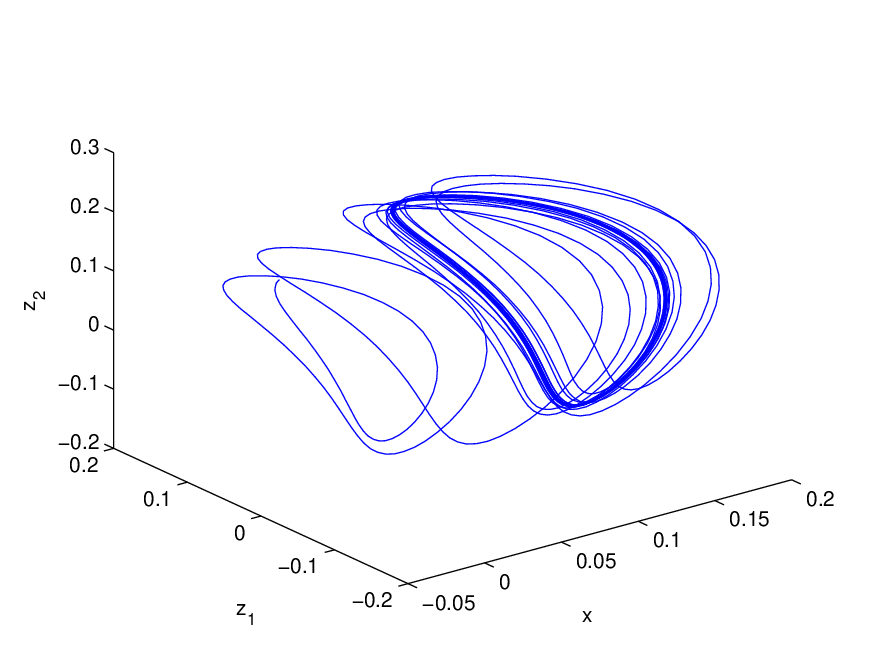}}
\subfigure[\label{D}Region C: An orbit approaching a stable equilibrium.]
{\includegraphics[width=.40\columnwidth,height=.23\columnwidth]{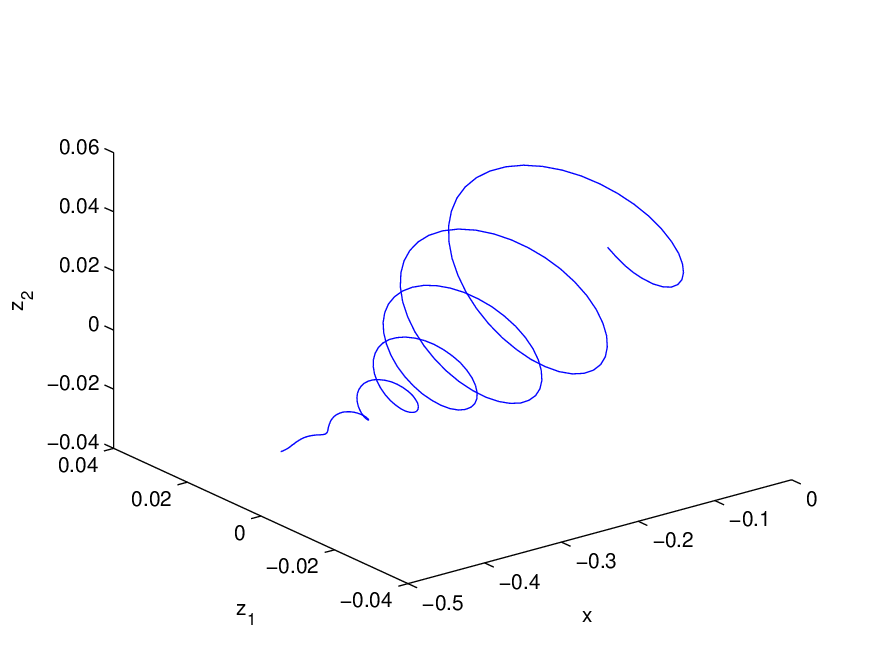}}
\caption{Figures \ref{C} and \ref{D} demonstrate a bistability: two orbits associated with the same parameters from region \(C\) in figure \ref{invTrans} approach to two different attractors. Orbits in \ref{C} and \ref{D} start with \((x:=0.01, z_1:=0.1, z_2:=0.1)\) and \((x:=-0.1, z_1:=0.01, z_2:=0.01),\) respectively. }\label{R2S3Orbits}
\end{center}
\end{figure}

\section*{Acknowledgment}

This research was in part supported by a grant from IPM (No. 94370421).

\end{document}